\documentclass[a4paper,11pt]{article}
\usepackage[utf8]{inputenc}
\usepackage[margin=1in]{geometry}

\usepackage[intlimits]{amsmath}
\usepackage{latexsym,amsfonts,amsthm,amssymb, epsfig}

\usepackage{color}
\usepackage{indentfirst}
\usepackage{mathrsfs}
\usepackage{leftidx}
\usepackage{bm}
\usepackage[english]{babel}
\usepackage{ulem}

\usepackage{cite}

\newcommand{\rn}{\mathbb{R}^n}
\newcommand{\real}{\mathbb{R}}
\newcommand{\nat}{\mathbb{N}}
\newcommand{\no}{\mathbb{N}_0}
\newcommand{\zn}{\mathbb{Z}^n}
\newcommand{\supp}{\mathop{\mathrm{supp}\,}\nolimits}

\newcommand{\p}{p(\cdot)}
\newcommand{\q}{q(\cdot)}
\newcommand{\s}{s(\cdot)}
\newcommand{\Sn}{\mathscr{S}(\rn)}
\newcommand{\SSn}{\mathscr{S}'(\rn)}
\renewcommand{\P}{\mathcal{P}(\rn)}
\newcommand{\Plog}{\mathcal{P}^{\log}(\rn)}
\newcommand{\SeqF}{L^{\phi}_{p(\cdot)}(\ell_{q(\cdot)})}
\newcommand{\SeqB}{\ell^{\phi}_{q(\cdot)}(L_{p(\cdot)})}
\newcommand{\Lp}{L_{\p}(\rn)}
\newcommand{\ellqp}{{\ell_{q(\cdot)}(L_{p(\cdot)})}}
\newcommand{\ellpq}{{L_{p(\cdot)}(\ell_{q(\cdot)})}}
\newcommand{\Bphi}{B^{\bm{w}, \phi}_{\p,\q}(\rn)}
\newcommand{\Fphi}{F^{\bm{w}, \phi}_{\p,\q}(\rn)}
\newcommand{\bphi}{b^{\bm{w}, \phi}_{\p,\q}(\rn)}
\newcommand{\fphi}{f^{\bm{w}, \phi}_{\p,\q}(\rn)}

\DeclareMathOperator*{\esssup}{ess\,sup}
\DeclareMathOperator*{\essinf}{ess\,inf}
\newtheorem{definition}{Definition }[section]
\newtheorem{theorem}[definition]{Theorem}

\newtheorem{corollary}[definition]{Corollary}
\newtheorem{lemma}[definition]{Lemma}
\theoremstyle{remark}\newtheorem{remark}[definition]{Remark}
\newtheorem{remarks}[definition]{Remarks}

\title{Non-smooth atomic decomposition of variable 2-microlocal Besov-type and Triebel-Lizorkin-type spaces}

\author{Helena F. Gon\c{c}alves \footnote{Institute of Mathematics, Friedrich-Schiller-University Jena, 07737 Jena, Germany; {\tt helena.goncalves@uni-jena.de}; the author was supported by the German Research foundation (DFG), Grant no. Ha 2794/8-11.}}

\begin{document}
\maketitle
\begin{abstract}
	In this paper we provide non-smooth atomic decompositions  of 2-microlocal Besov-type and Triebel-Lizorkin-type spaces with variable exponents $\Bphi$ and $\Fphi$. Of big importance in general, and an essential tool here, are the characterizations of the spaces via maximal functions and local means, that we also present. 
	These spaces were recently introduced by Wu at al. and cover not only variable 2-microlocal Besov and Triebel-Lizorkin spaces $B^{\bm{w}}_{\p,\q}(\rn)$ and $F^{\bm{w}}_{\p,\q}(\rn)$, but also the more classical smoothness Morrey spaces  $B^{s, \tau}_{p,q}(\rn)$ and $F^{s,\tau}_{p,q}(\rn)$. Afterwards, we state a pointwise multipliers assertion for this scale.
\end{abstract}

\section{Introduction}
The introduction of function spaces with variable integrability, also known as variable exponent function spaces $\Lp$, goes back to Orlicz \cite{Or31} in 1931. However, only several decades later they were substantially studied, in the papers \cite{KR91} of Kov\'a\v{c}ik and R\'akosn\'ik, as well as \cite{ER00} of Edmunds and R\'akosn\'\i k and \cite{Die04} of Diening. The spaces $L_{\p}(\rn)$ have several applications, such as in fluid dynamics, image processing, PDEs and variational calculus. For an overview we refer to \cite{DHHR}.

The merger of the concepts of variable integrability and variable smoothness was done by Diening, H{\"a}st{\"o} and Roudenko in \cite{DHR09}, where the authors defined Triebel-Lizorkin spaces with variable exponents $F^{s(\cdot)}_{\p, \q}(\rn)$. The interplay between the three parameters $s, q$ and $p$ can, easily and interestingly, be verified on the trace theorem  on $\mathbb{R}^{n-1}$ proved by these authors \cite[Theorem~3.13]{DHR09}. This interaction is also clear on the Sobolev embedding results obtained for these spaces by Vyb\'iral in \cite{Vyb09}. 

For the Besov spaces it is not so easy to have also the parameter $q$ as a variable one. Almeida and H{\"a}st{\"o} introduced in \cite{AH10} Besov spaces $B_{\p,\q}^{s(\cdot)}(\rn)$ with all three indices variable, using for that a different modular which already uses the variable structure on $\q$. They proved the Sobolev and other usual embeddings in this scale.

A more general approach to spaces of variable smoothness are the so-called 2-microlocal function spaces, where the smoothness gets measured by a weight sequence $\bm{w}=(w_j)_{j\in \nat_0}$. Besov spaces with such weight sequences appeared first in the works of Peetre \cite{Pee75} and Bony \cite{Bony}. The variable 2-microlocal Besov and Triebel-Lizorkin spaces $B^{\bm{w}}_{\p,q}(\rn)$ and $F^{\bm{w}}_{\p,\q}(\rn)$ were introduced by Kempka in \cite{Kem09, Kem10}. Since then, several authors have devoted some attention to these spaces, expanding the knowledge about their properties. We mention \cite{AC16, AC16-1, GMN14, GK16, GK17, GKV17, Kem16, KV12, MNS13}. 

Function spaces with variable exponents represent a kind of approach that generalizes classical function spaces. However, there are different approaches, which also lead us to generalized Besov and Triebel-Lizorkin spaces. The Besov-type spaces ${B}_{p,q}^{s,\tau}(\rn)$ and the Triebel-Lizorkin-type spaces ${F}_{p,q}^{s,\tau}(\rn)$ are an example of that. They were introduced in \cite{YSY10} and, besides the classical Besov and Triebel-Lizorkin spaces, they also cover Triebel-Lizorkin-Morrey spaces introduced by Tang and Xu in \cite{TX05} and the hybrid functions spaces introduced and studied by Triebel in \cite{t13,t14}, together with their use in heat and Navier-Stokes equations.

Recently also these scales got new variable versions. In \cite{YZY15-B,YZY15-F}, the authors introduced Besov-type and Triebel-Lizorkin-type spaces with variable exponents $B_{\p,\q}^{\s,\phi}(\rn)$ and $F_{\p,\q}^{\s,\phi}(\rn)$, with $\phi$ being a measurable function on $\real^{n+1}_+$ which replaces the parameter $\tau$. Moreover, also the 2-microlocal versions $\Bphi$ and $\Fphi$ were already introduced in \cite{WYY18}. Among other aspects, the authors characterized the spaces by means of atomic decompositions and obtained a trace result on hyperplanes. These spaces provide an unified approach that covers variable 2-microlocal Besov and Triebel-Lizorkin spaces $B^{\bm{w}}_{\p,\q}(\rn)$ and $F^{\bm{w}}_{\p,\q}(\rn)$, variable Besov-type and Triebel-Lizorkin-type spaces $B_{\p,\q}^{\s,\phi}(\rn)$ and $F_{\p,\q}^{\s,\phi}(\rn)$, and hence all the spaces that are already covered by these.

In this paper we aim to derive a non-smooth atomic characterization for the spaces $\Bphi$ and $\Fphi$. An essential tool here is their characterization via local means, which follows immediately from the characterization by maximal functions. Although this characterization was already considered in \cite{WYY18}, now we prove a more general version, which is more in line with the results of this type present in the literature. 

As for the characterization via non-smooth atoms, recently in \cite{GM18} the authors proved a result for the spaces $F_{\p,\q}^{\s,\phi}(\rn)$, which was the first result on this subject for this type of function spaces, even for the case of constant exponents. Now we extend it to the 2-microlocal spaces $\Fphi$ and also complete this study by obtaining the counterpart for the Besov scale $\Bphi$. Covered by these results will be also the results obtained in \cite{GK16} for 2-microlocal variable Besov and Triebel-Lizorkin spaces. 

Implicit in the name of this characterization -- \textit{non-smooth} -- is the fact that we replace the usual (smooth) atoms by more general ones, in the sense that they have weaker assumptions on the smoothness. We then show that, also in this case, all the crucial information comparing to smooth atomic decompositions is kept. This modification appeared first in \cite{TW96}, where Triebel and Winkelvoß suggested the use of these more relaxed conditions to define classical Besov and Triebel-Lizorkin spaces intrinsically on domains. More recent is the work \cite{Sch13} of Scharf, where a non-smooth atomic characterization for $B^s_{p,q}(\rn)$ and $F^s_{p,q}(\rn)$ was derived, using even more general atoms. Here we follow this approach to prove our main result. Moreover, as an application, we provide an assertion on pointwise multipliers for the spaces $\Bphi$ and $\Fphi$.

\section{Notation and definitions}

We start by collecting some general notation used throughout the paper.

As usual,  we denote by $\nat$ the set of all natural numbers, $\nat_0=\mathbb N\cup\{0\}$, and 
$\rn$, $n\in\nat$,  the $n$-dimensional real Euclidean space with $|x|$, for $x\in\rn$, denoting the Euclidean norm of $x$. 
By $\zn$ we denote the lattice of all points in $\rn$ with integer components.
For $\beta:=(\beta_1,\cdots, \beta_n)\in \mathbb{Z}^n$, let $|\beta|:= |\beta_1|+\cdots+|\beta_n|$.
If $a,b\in\real$,  then $a\vee b:=\max \{a,b\}$.
% $a_+:=\max(a,0)$ and let $[a]$ denote its integer part. 
%For $p\in (0,\infty]$, the number $p'$ is defined by
%$1/p':=(1-1/p)_+$ with the convention that $1/\infty=0$. 
We denote by  $c$ a generic positive constant which is independent of the main parameters, but its value may change from line to line. 
The expression $A\lesssim B$ means that $ A \leq c\,B$. If $A \lesssim
B$ and $B\lesssim A$, then we write $A \sim B$.  

Given two quasi-Banach spaces $X$ and $Y$, we write $X\hookrightarrow Y$ if $X\subset Y$ and the natural embedding is bounded.

If $E$ is a measurable  subset of $\rn$, we denote by $\chi_E$ its characteristic function and by  $|E|$ its Lebesgue measure. By $\supp f$ we denote the support of the function $f$.

For each cube $Q\subset \rn $ we denote its center by $c_Q$ and its side length by $ \ell(Q)$ and, for $a\in (0,\infty)$ we denote by $aQ$ the cube concentric with $Q$ having the side length $a\ell(Q)$. For $x\in\rn$ and $r \in (0, \infty)$, we denote by $Q(x, r)$ the cube centered at $x$ with side lenght $r$, whose sides are parallel to the axes of coordinates. 

Given $k\in\no$, $C^k(\rn)$ is  the space of all functions $f:\rn\rightarrow \mathbb{C}$ which are $k$-times continuously differentiable (continuous in $k=0$) such that 
$$
\|f\mid C^k(\rn)\|:=\sum_{|\alpha|\leq k} \sup_{x\in\rn}|D^{\alpha}f(x)|<\infty.
$$
 The H\"older space $ \mathscr{C}^s(\rn)$ with index $s>0$ is defined as the set of all functions $f\in C^{{\lfloor s \rfloor}^-}(\rn)$ with 
$$
\|f\mid  \mathscr{C}^s(\rn)\|:=\|f\mid C^{{\lfloor s \rfloor}^-}(\rn)\|+ \sum_{|\alpha|=\{s\}^+}\sup_{x,y\in\rn, x\neq y} \frac{|{D^{\alpha}}f(x)-{D^{\alpha}}f(y)|}{|x-y|^{\{s\}^+}}<\infty,
$$
where  ${\lfloor s \rfloor}^-\in\no$ and $\{s\}^+\in(0,1]$  are uniquely determined numbers so that $s={\lfloor s \rfloor}^-+\{s\}^+$.  If $s=0$ we set $\mathscr{C}^0(\rn):=L_{\infty}(\rn)$.
%It holds $C^k(\rn)\hookrightarrow \mathscr{C}^k(\rn)$ for $k\in\no$.

By $\Sn$ we denote  the usual Schwartz  class of all infinitely differentiable  rapidly decreasing complex-valued  functions on $\rn$ and $\SSn$ stands for the dual space of tempered distributions. The Fourier transform of $f\in \Sn$ 
or $f\in \SSn$ is denoted by $\widehat{f}$, while its inverse transform is denoted by $f^{\vee}$.

\medskip

Now we give a short survey on variable exponents. For a measurable function $p:\rn\rightarrow (0,\infty]$, let  
$$
p^-:= \essinf_{x\in \rn}p(x) \quad \text{and} \quad  p^+:= \esssup_{x\in \rn}p(x).
$$
In this paper we denote by  $\P$  the set  of all measurable functions $p:\rn\rightarrow (0,\infty ]$ (called variable exponents) which are essentially bounded away from zero. For $p\in \P$ and a measurable set $E\subset\rn$, the space $L_{p(\cdot)}(E)$ is defined to be the set of all (complex or real-valued) measurable functions $f$ such that  
$$
\|f \mid L_{p(\cdot)}(E)\|:=\inf\bigg\{\lambda\in (0,\infty):\int_{E} \biggl(\frac{|f(x)|}{\lambda}\biggr)^{p(x)}\leq 1 \biggr\}<\infty.
$$
It is known that $L_{\p}(E)$ is a quasi-Banach space, a Banach space when $p^-\geq 1$. If $\p\equiv p$ is constant, then $L_{\p}(E)=L_p(E)$ is the classical Lebesgue space. 

For later use we recall that $L_{\p}(E)$ has the lattice property. Moreover, we have 
$$
\|f\mid L_{\p}(E)\|= \big\| |f|^r\big| L_{\frac{\p}{r}}(E)\big\|^{\frac 1r},  \qquad r\in (0,\infty),
$$
and 
$$
\|f+g\mid L_{\p}(E)\|^r\leq \|f\mid L_{\p}(E)\|^r+\|g\mid L_{\p}(E)\|^r, \qquad r\in \bigl(0,\min\{1,p^-\}\bigr).
$$

\medskip

In the setting of variable exponent function spaces it is needed to require some regularity conditions to the exponents. We recall now the standard conditions used.

\begin{definition}\label{log-holder}
	Let $g\in C(\rn)$. We say that $g$ is locally log-H\"older continuous, abbreviated $g\in C_{\rm{loc}}^{\log}(\rn)$, if there exists $c_{\log}(g)>0$ such that
	\begin{equation} \label{loc-log-holder}
	|g(x)-g(y)|\leq \frac{c_{\log}(g)}{\log(e+1/|x-y|)} \quad \text{for all}\;\; x,y\in\rn.
	\end{equation}
	We say that $g$ is globally log-H\"older continuous, abbreviated $g\in C^{\log}(\rn)$, if $g$ is locally log-H\"older continuous and there exists $g_{\infty}\in\real$ such that
	\begin{equation}\label{glob-log-holder}
	|g(x)-g_{\infty}|\leq \frac{c_{\log}}{\log(e+|x|)} \quad \text{for all}\;\; x\in\rn.
	\end{equation}
\end{definition}
\vspace{0.5cm}

Note that all functions in $C_{\rm{loc}}^{\log}(\rn)$ are bounded and if $g\in C^{\log}(\rn)$ then $g_{\infty}= \lim_{|x|\rightarrow \infty}g(x)$. Moreover, for $g \in \mathcal{P}(\rn)$ with $g_+<\infty$, we have that $g\in C^{\log}(\rn)$ if, and only if, $1/g \in C^{\log}(\rn)$. %\marginpar{\sdm{esta equivalência não é só quando $g_+<\infty$?}}
The notation $\mathcal{P}^{\log}(\rn)$ is used for those variable exponents $p\in \mathcal{P}(\rn)$ with $p\in C^{\log}(\rn)$.\\

\subsection{Mixed sequence-Lebesgue spaces}
We introduce now mixed sequence-Lebesgue spaces and, in the next subsection, we present some properties about them, which will be very useful throughout this work.
\begin{definition} \label{def:mixed} Let $p,q \in \P$ and $E$ be a measurable subset of $\rn$. 
	\begin{list}{}{\labelwidth1.3em\leftmargin2.3em}
		\item[{\upshape (i)\hfill}] Let $p^+, q^+<\infty$. The space $L_{\p}(\ell_{\q}(E))$ is defined to be the set of all sequences of measurable functions $(f_j)_{j \in \nat_0}$ on $E$ such that 
		\begin{equation}\label{eq:defLpq}
		\| (f_j)_{j \in \nat_0} \mid L_{\p}(\ell_{\q}(E)) \| := \left\| \left(\sum_{j=0}^{\infty} |f_{j}(x)|^{q(x)} \right)^{1/q(x)} \mid L_{\p}(E)\right\|<\infty. 
		\end{equation}
		\item[{\upshape (ii)\hfill}] The space $\ell_{\q}(L_{\p}(E))$ is defined to be the set of all sequences of measurable functions $(f_j)_{j \in \nat_0}$ on $E$ such that
		\begin{equation} \label{eq:norm-B}
		\|(f_j)_{j \in \nat_0} \mid \ell_{\q}(L_{\p}(E)) \| := \inf \left\{ \mu>0 : \varrho_{\ell_{\q}(L_{\p})}\left(\left(\frac{f_{j} \chi_E}{\mu}\right)_{j \in \nat_0}\right) \leq 1\right\} < \infty, 
		\end{equation}
		where, for all sequences $(g_j)_{j \in \nat_0}$ of measurable functions, 
		\begin{equation} \label{eq:modular-B-original}
		\varrho_{\ell_{\q}(L_{\p})}\big((g_j)_{j \in \nat_0}\big) := \sum_{j=0}^{\infty} \inf\left\{\lambda_{j}>0 : \varrho_{\p}\left(\frac{g_{j}}{\lambda_{j}^{1/\q}} \right)\leq 1 \right\},
		\end{equation}
		with the convention $\lambda^{1/\infty}=1$ for all $\lambda \in (0, \infty)$. 
	\end{list}
\end{definition}

\begin{remark}\label{rmk:mixed}
	\begin{list}{}{\labelwidth1.3em\leftmargin2.3em}
		\item[{\upshape (i)\hfill}]  If $E= \rn$, we simply write $ L_{\p}(\ell_{\q})$ or ${\ell_{\q}(L_{\p})}$.
		\item[{\upshape (ii)\hfill}] If $p,q \in \P$, then $\| \cdot \mid L_{\p}(\ell_{\q}) \|$ and $\| \cdot \mid {\ell_{\q}(L_{\p})}\|$  are quasi-norms in $L_{\p}(\ell_{\q})$ and ${\ell_{\q}(L_{\p})}$, respectively. Moreover, $\| \cdot \mid L_{\p}(\ell_{\q}) \|$ is a norm if $\min\{p^-, q^-\}\geq 1$. The same does not hold when it comes to the norm in \eqref{eq:norm-B}. It was shown in \cite{AH10} that $\| \cdot \mid \ell_{\q}(L_{\p})\|$ is a norm either when $q\geq 1$ is constant and $p^-\geq 1$, or when $\displaystyle \frac{1}{p(x)}+\frac{1}{q(x)}\leq 1$ almost everywhere. More recently, it was proved in \cite{KemVybnorm} that it also becomes a norm if $1\leq q(x)\leq p(x)\leq \infty$. 
		\item[{\hfill (iii)\hfill}] It was proved in \cite[Theorem~3.5]{AH10} that $\varrho_{\ell_{\q}(L_{\p})}$ is a semi-modular. The left-continuity property ensures that the \textit{unit ball property} holds, i.e., 
			\begin{equation}
				\| (f_j)_{j \in \nat_0} \mid \ell_{\q}(L_{\p}) \| \leq 1\quad \mbox{if, and only if,}\quad \varrho_{\ell_{\q}(L_{\p})}\big((f_j)_{j \in \nat_0}\big) \leq 1.\nonumber
			\end{equation}
		{Moreover, for $r\in(0,\infty)$,
		$$	\| (f_j)_{j \in \nat_0} \mid \ell_{\q}(L_{\p}) \|=
			\| (|f_j|^r)_{j \in \nat_0} \mid \ell_{\frac{\q}{r}}(L_{\frac{\p}{r}}) \|^{1/r}.$$}	
		\item[{\upshape (iv)\hfill}] If $q^+ <\infty$, then we can replace \eqref{eq:modular-B-original} by the simpler expression
		\begin{equation} \label{eq:modular-B}
		\varrho_{\ell_{\q}(L_{\p})}\big((g_j)_{j \in \nat_0}\big) = \sum_{j=0}^{\infty} \Big\| |g_{j}|^{\q} \mid L_{\frac{\p}{\q}} (\rn)\Big\|. 
		\end{equation}
	\end{list}
\end{remark}

\medskip

Let $\mathcal{G}(\real^{n+1}_+)$ be the set of all measurable functions $\phi: \rn \times [0, \infty) \rightarrow (0, \infty)$ having the following properties:
there exist positive constants $c_1(\phi)$ and $\tilde{c_1}(\phi)$ such that
\begin{equation} \label{phi-cond1}
\frac{1}{\tilde{c_1}(\phi)} \leq \frac{\phi(x,r)}{\phi(x, 2r)} \leq c_1(\phi) \quad \text{ for all } x\in\rn \;\, \text{and} \;\,   r\in(0,\infty),
\end{equation}
and there exists a positive constant $c_2(\phi)$ such that, for all $x, y \in \rn$ and $r \in (0, \infty)$ with $|x-y|\leq r$,
\begin{equation} \label{phi-cond2}
\frac{1}{c_2(\phi)} \leq \frac{\phi(x,r)}{\phi(y, r)} \leq c_2(\phi).
\end{equation}
The conditions \eqref{phi-cond1} and  \eqref{phi-cond2} are called doubling condition and compatibility condition, respectively, and have been used by Nakai \cite{Nak93, Nak06} and Nakai and Sawano \cite{NS12}. 
Examples of functions in $\mathcal{G}(\real^{n+1}_+)$ are provided in \cite[Remark~1.3]{YZY15-F}.

In what follows, for $\phi \in \mathcal{G}(\real^{n+1}_+)$ and {a cube} $Q:= Q(x, r)$ with center $x \in \rn$ and  {side length} $r \in (0, \infty)$, we define $\phi (Q) := \phi (Q(x,r)) := \phi(x, r)$.

%For all $x \in \rn$ and $r \in (0, \infty)$, denote by $Q(x,r)$ the cube centered at $x$ with side-length $r$, whose sides are parallel to the axes of coordinates. Let $\phi: \rn \times [0, \infty) \rightarrow (0, \infty)$ be a measurable function. We always assume that $\phi$ satisfies the following two conditions: 
%\begin{description}
%\item[(S1)] there exist positive constants $c_1$ and $\tilde{c_1}$ such that, for all $x \in \rn$ and $r \in (0, \infty)$, 
%$$
%\frac{1}{\tilde{c_1}} \leq \frac{\phi(x,r)}{\phi(x, 2r)} \leq c_1;
%$$
%\item[(S2)] there exists a positive constant $c_2$ such that, for all $x, y \in \rn$ and $r \in (0, \infty)$ with $|x-y|\leq r$, 
%$$
%\frac{1}{c_2} \leq \frac{\phi(x,r)}{\phi(y, r)} \leq c_2.
%$$
%\end{description}
%In what follows, for all cubes $Q:= Q(x, r)$ with $x \in \rn$ and $r \in (0, \infty)$, let 
%$$
%\phi (Q) := \phi (Q(x,r)) := \phi(x, r). 
%$$

%For $j \in \mathbb{Z}$ and $k \in \zn$, denote by $Q_{jk}$ the dyadic cube $2^{-j}([0,1)^n + k)$, $x_{Q_{jk}}:= 2^{-j}k$ its lower left corner and $\ell(Q_{jk})$ its side length. Let
%$$
%\mathcal{Q}:= \{ Q_{jk} : j \in \mathbb{Z}, k \in \zn\}, \quad \mathcal{Q}^* := \{Q \in \mathcal{Q}: \ell(Q)\leq 1\}
%$$
%and $j_Q:= - \log_2 \ell(Q)$ for all $Q \in \mathcal{Q}$.

The following is our convention for dyadic cubes:  For $j \in \mathbb{Z}$ and $k \in \zn$, denote by $Q_{jk}$ the dyadic cube $2^{-j}([0,1)^n + k)$ and $x_{Q_{jk}}$ its lower left corner. Let  $\mathcal{Q}:= \{ Q_{jk} : j \in \mathbb{Z}, \,k \in \zn\}$, $\mathcal{Q}^*:= \{ Q\in\mathcal{Q} : \ell(Q)\leq 1\}$ and  $j_Q:= - \log_2 \ell(Q)$  for all $Q\in  \mathcal{Q}$. 
When the dyadic cube $Q$ appears as an index, such as $\sum_{Q\in \mathcal{Q}}$ and $(\cdots)_{Q\in \mathcal{Q}}$, it is understood that $Q$ runs over all dyadic cubes in $\rn$. 

\medskip

For the function spaces under consideration in this paper, the following modified mixed-Lebesgue sequence spaces are of special importance. 
\begin{definition}\label{mixed-phi}
	Let $p,q\in \mathcal{P}(\rn)$ and $\phi\in \mathcal{G}(\real^{n+1}_+)$.
	\begin{list}{}{\labelwidth1.3em\leftmargin2.3em}
		\item[{\upshape (i)\hfill}] We denote by $\SeqB$ the set of all sequences $(g_j)_{j\in\nat_0}$ of measurable functions on $\rn$ such that 
		$$
			\| (g_j)_{j\in\nat_0} \mid \SeqB \|:= \sup_{P\in \mathcal{Q}} \frac{1}{\phi(P)} \| (g_j)_{j\geq(j_p\vee 0)} \mid \ell_{\q}(L_{\p}(P))\|<\infty,
		$$
		where  the supremum is taken over all dyadic cubes $P$ in $\rn$.  
		\item[{\upshape (ii)\hfill}] We denote by $\SeqF$ the set of all sequences $(g_j)_{j\in\no}$ of measurable functions on $\rn$ such that 
		$$
		\| (g_j)_{j\in\nat_0} \mid \SeqF \|:= \sup_{P\in \mathcal{Q}} \frac{1}{\phi(P)} \| (g_j)_{j\geq(j_p\vee 0)} \mid L_{\p}(\ell_{\q}(P))\|<\infty,
		$$
		where  the supremum is taken over all dyadic cubes $P$ in $\rn$. 
		\end{list}
\end{definition}
 
\begin{remark} We remark that $\SeqB$ and $\SeqF$ are quasi-normed spaces that coincide with the mixed Lebesgue-sequence spaces $\ell_{\q}(L_{\p})$ and $L_{\p}(\ell_{\q})$ from Definition \ref{def:mixed}, respectively, when $\phi\equiv 1$. The case of $\q=q$ constant and $\phi(Q):=|Q|^{\tau}$ for all cubes $Q$ and $\tau\in[0,\infty)$, has also been considered in \cite{LYYSU13}.\\
\end{remark}

\subsection{Auxiliary results}

Although the Hardy-Littlewood maximal operator $\mathcal{M}_t$ constitutes a great tool in the theory of classical function spaces and also in the scale of variable Lebesgue spaces, it is not, in general, a good instrument in the mixed spaces $\ellpq$ and $\ellqp$. It was actually proved in \cite{AH10} and in \cite{DHR09} that this operator is not bounded in these spaces if one considers $q$ non-constant. However, this adversity can be overcome by the use of convolution inequalities involving radially decreasing kernels, namely the so-called \textit{$\eta$-functions}, defined by
$$ \eta_{\nu, R}(x) := \frac{2^{n \nu }}{(1+2^{\nu }|x|)^R}, \quad x\in \rn,$$
for $\nu \in \nat_0$ and $R>0$. \\%In particular, for $R>n$ we have $\eta_{\nu,R}\in L_1(\rn)$ and \linebreak $\|\eta_{\nu,R}\mid L_1(\rn)\|=c_{R,n}$ is independent of $\nu$.

The next result was proved in \cite[Lemma~4.6.3]{DHHR} and shows that the convolution operator is well-behaved in $\Lp$ for $p \in \Plog$, when considering radially decreasing integrable functions. 
\begin{lemma}
	Let $p \in \Plog$ with $p(x)\geq 1$. Let $\psi \in L_1(\rn)$ and $\psi_{\varepsilon}(x):= \varepsilon^{-n}\psi(x/\varepsilon)$, for $\varepsilon>0$. Suppose that $\Psi(x):=\sup_{|y|\geq |x|}|\psi(y)|$ is integrable and $f \in \Lp$. Then
	\begin{equation*}
	\|\psi_\varepsilon \ast f \mid \Lp\| \lesssim \| \Psi \mid L_1(\rn)\| \|f \mid \Lp\|, 
	\end{equation*}
	where the implicit constant depends only on $n$ and $p$. 
\end{lemma}

\begin{remark}\label{rmk:conv-Lp} Note that we can use the previous lemma for the $\eta$-functions defined above. Namely, taking $\psi=\eta_{0, m}$ with $m>n$, then we have $\Psi= \eta_{0,m} \in L_1(\rn)$. Thus, setting $\psi_\varepsilon=\eta_{\nu,m}$ with $\varepsilon = 2^{-\nu}$, we get 
	$$
	\| \eta_{\nu,m} \ast f \mid \Lp\| \lesssim  \| f \mid \Lp\|
	$$
	if $m>n$, for $f \in \Lp$ and $p \in \Plog$ with $p(x)\geq 1$, for $x \in \rn$.
\end{remark}

The following two results show that the $\eta$-functions are well suited for the mixed Lebesgue-sequence spaces. The first one was proved in \cite[Theorem~3.2]{DHR09} and the second goes back to \cite[Lemma~10]{KV12}.

\begin{lemma}\label{conv-F} Let $p, q \in \Plog$ with $1<p^-\leq p^+ <\infty$ and $1<q^-\leq q^+ <\infty$. Then the inequality
	\begin{equation}
	\| \left(\eta_{\nu, R} \ast f_{\nu} \right)_{\nu \in \nat_0}\mid \ellpq \| \leq c\, \|(f_{\nu})_{\nu \in \nat_0} \mid \ellpq \| \nonumber
	\end{equation}
	holds for every sequence $(f_{\nu})_{\nu \in \nat_0}$ of $L_1^{\rm{loc}}(\rn)$ functions and constant $R>n$.
\end{lemma}

\begin{lemma}\label{conv-B} Let $p, q \in \Plog$ with $p^- \geq 1$. For all $R>n +c_{\log}(1/q)$, there exists a constant $c>0$ such that for all sequences $(f_{\nu})_{\nu \in \nat_0} \in \ell_{\q}(L_{\p})$ it holds
	\begin{equation}
	\| \left(\eta_{\nu, R} \ast f_{\nu} \right)_{\nu \in \nat_0} \mid \ellqp \| \leq c \, \|(f_{\nu})_{\nu \in \nat_0} \mid \ellqp \|. \nonumber
	\end{equation}
\end{lemma}
In the next result we state the corresponding counterparts for the modified mixed Lebesgue sequence spaces from Definition \ref{mixed-phi}. %A primary version of it was described in \cite[Lemma~3.12]{WYY18}, but with conditions stricter than it would be desirable. Therefore we proceed differently 

%Sometimes we need to estimate norms of characteristic functions on cubes. The following result it was proved in \cite[Corollary~4.5.9]{DHHR}.
%\begin{lemma}
%	Let $p \in \Plog$ with $p(x)\geq 1$. Then 
%	\begin{equation}
%	\| \chi_Q\mid \Lp\| \sim \begin{cases}
%	|Q|^{\frac{1}{p(x)}}, & \text{ if } \ell(Q)\leq 1 \text{ and } x \in Q,\\
%	|Q|^{\frac{1}{p_\infty}}, & \text{ if } \ell(Q)\geq 1,
%	\end{cases} \nonumber
%	\end{equation}
%	for every cube $Q\subset \rn$. The implicit constants only depend on $c_{\log}(p)$. 
%\end{lemma}

\begin{lemma} \label{eta-phi}
	Let $p, q \in \Plog$ and $\phi\in \mathcal{G}(\real^{n+1}_+)$. 
	\begin{list}{}{\labelwidth1.3em\leftmargin2.3em}
		\item[{\upshape (i)\hfill}] Let $p^- \geq 1$. If
		$$R > n + c_{\log}(1/q) + \max\left\{0, \log_2 \tilde{c}_1(\phi)\right\},$$
		then there exists $c>0$ such that {for all sequences $(f_{\nu})_{\nu \in \nat_0} \in \SeqB$ it holds}
		\begin{equation}
		\| \left(\eta_{\nu, R} \ast f_{\nu} \right)_{\nu \in \nat_0} \mid \SeqB \| \leq c \, \|(f_{\nu})_{\nu \in \nat_0} \mid \SeqB \|. \nonumber
		\end{equation}
		\item[{\upshape (ii)\hfill}] Let $1\leq p^-\leq p^+ <\infty$ and $1\leq q^-\leq q^+ <\infty$. If 
		$$R > n + \max\left\{0, \log_2 \tilde{c}_1(\phi)\right\},$$
		then there exists $c>0$ such that {for all sequences $(f_{\nu})_{\nu \in \nat_0} \in \SeqF$ it holds}
		\begin{equation}
		\| \left(\eta_{\nu, R} \ast f_{\nu} \right)_{\nu \in \nat_0} \mid \SeqF \| \leq c \, \|(f_{\nu})_{\nu \in \nat_0} \mid \SeqF \|. \nonumber
		\end{equation}
	\end{list}
\end{lemma}
\begin{proof}
	We will prove part (i), as the second follows similarly.
	
For any given dyadic cube $P \in \mathcal{Q}$ and any $\nu \in \nat_0$, we decompose each $f_\nu$ into the sum
$$
f_\nu = f_\nu^0 + \sum_{i=1}^\infty f_\nu^i,
$$
where
$$
f_\nu^0:= f_\nu \, \chi_{Q(c_P, 2^{-j_P+1})} \quad \mbox{and} \quad f_\nu^i:= f_\nu \, \chi_{R_i}, \quad \mbox{with } R_i:= Q(c_P, 2^{-j_P+i+1})\backslash Q(c_P, 2^{-j_P+i}),
$$
for $i \in \nat$ and $c_P$ being the center of the cube $P$. Then we have
\begin{align*}
&\frac{1}{\phi(P)}\| \left(\eta_{\nu, R} \ast f_{\nu} \right)_{\nu \geq (j_P \vee 0)} \mid  \ell_{\q}(L_{\p}(P)) \| \\
&\leq \frac{1}{\phi(P)} \| \left(\eta_{\nu, R} \ast f_{\nu}^0 \right)_{\nu \geq (j_P \vee 0)} \mid \ell_{\q}(L_{\p}(P)) \| + \frac{1}{\phi(P)} \Big\| \Big(\sum_{i=1}^{\infty }\eta_{\nu, R} \ast f_{\nu}^i \Big)_{\nu \geq (j_P \vee 0)} \mid  \ell_{\q}(L_{\p}(P)) \Big\|\\
&=: I_1 + I_2. 
\end{align*}
	 
	 We claim that, for $j=1,2$,
	\begin{equation}\label{claim}
	 I_j \, \lesssim  \|(f_{\nu})_{\nu \in \nat_0} \mid \SeqB \|, 
	 %I_j \, \lesssim \, \sup_{R\in \mathcal{Q}} \frac{1}{\phi(R)} \| \left(f_{\nu} \right)_{\nu \geq (j_R \vee 0)} \mid  \ell_{\q}(L_{\p}(R)) \| =: \|(f_{\nu})_{\nu \in \nat_0} \mid \SeqB \|, 
	\end{equation}
	 which, together with the arbitrariness of $P \in \mathcal{Q}$, allows us to conclude the proof. \\
	 
	 {\it Step 1.} Let us prove \eqref{claim} for $j=1$. Here we apply Lemma \ref{conv-B} with $R>n + c_{\log}(1/q)$ and use \eqref{phi-cond1} to obtain
	 \begin{align*}
	 I_1 &= \frac{1}{\phi(P)} \| \left(\eta_{\nu, R} \ast f_{\nu}^0 \right)_{\nu \geq (j_P \vee 0)} \mid \ell_{\q}(L_{\p}(P)) \| \\
	 &\leq \frac{1}{\phi(P)} \| \left( f_{\nu}^0 \right)_{\nu \geq (j_P \vee 0)} \mid \ell_{\q}(L_{\p}(\rn)) \| \\
	 & \leq \tilde{c}_1(\phi) \, \frac{1}{\phi(Q(c_P, 2^{-j_p+1}))} \big\| \left( f_\nu \, \chi_{Q(c_P, 2^{-j_P+1})} \right)_{\nu \geq (j_P \vee 0)} \mid \ell_{\q}(L_{\p}(\rn)) \big\|\\
	 &\lesssim \|(f_{\nu})_{\nu \in \nat_0} \mid \SeqB \|,
	 \end{align*}
	 as we desired. \\
	 
	 {\it Step 2.} To estimate $I_2$, we start by estimating the convolution appearing here. Note that, for $x \in P$, $y \in R_i$, $i \in \nat$, and $\nu \geq j_P$, we have $|x-y|\geq 2^{i-1-j_P}\geq 2^{i-1-\nu}$ and hence
	 \begin{align*}
	  (\eta_{\nu,R} \ast f_\nu^i)(x) &= \int_{\rn} \frac{2^{\nu n}}{(1+2^\nu|x-y|)^R} \, f_\nu^i(y) \, dy \, \lesssim \, 2^{-i \varepsilon} \, (\eta_{\nu, R-\varepsilon} \ast f_\nu^i)(x),
	 \end{align*}
	 for some constant $\varepsilon$ satisfying $\varepsilon > \max\{0, \log_2 \tilde{c}_1(\phi)\}$ and $R-\varepsilon> n + c_{\log}(1/q)$. Therefore, using this, Lemma \ref{conv-B} and \eqref{phi-cond1}, we get
	 \begin{align*}
	 I_2 & \lesssim \frac{1}{\phi(P)} \sum_{i=1}^\infty 2^{-i \varepsilon} \, \big\| (\eta_{\nu, R- \varepsilon} \ast f_{\nu}^i)_{\nu \geq (j_P \vee 0)} \mid  \ell_{\q}(L_{\p}(P)) \big\|\\
	 &\leq  \frac{1}{\phi(P)} \sum_{i=1}^\infty 2^{-i \varepsilon} \, \big\| ( f_{\nu}^i)_{\nu \geq (j_P \vee 0)} \mid  \ell_{\q}(L_{\p}(\rn)) \big\|\\
	 &\leq \sum_{i=1}^\infty 2^{-i \varepsilon}\,  \frac{\phi(Q(c_P, 2^{-j_p+i+1}))}{\phi(P)}  \|(f_{\nu})_{\nu \in \nat_0} \mid \SeqB \|\\
	 &\leq \|(f_{\nu})_{\nu \in \nat_0} \mid \SeqB \|\, \sum_{i=1}^\infty 2^{-i (\varepsilon- \log_2\tilde{c}_1(\phi))}\\
	 &\lesssim \|(f_{\nu})_{\nu \in \nat_0} \mid \SeqB \|,
	 \end{align*}	
	 which allows us to conclude the proof of \eqref{claim}. Consequently the proof of Lemma \ref{eta-phi} is complete. 
%	{\it Step 1.} We will first consider the case when $q^+<\infty$. Let us assume that $\|(f_{\nu})_{\nu \in \nat_0} \mid \SeqF \|=1$. By the unit ball property stated in Remark \ref{rmk:mixed}(iii), we then know that 
%	$$
%	\sum_{\nu=(j_R\vee 0)}^\infty \bigg\| \Big|\frac{f_\nu \, \chi_R}{\phi(R)}\Big|^{\q} \mid L_{\frac{\p}{\q}}(\rn)\bigg\| \leq 1,
%	$$
%	for any cube $R \in \mathcal{Q}$. We aim to prove that $\| \left(\eta_{\nu, R} \ast f_{\nu} \right)_{\nu \in \nat_0} \mid \SeqF \| \leq 1$. Fix $P \in \mathcal{Q}$ arbitrarily. We prove that this follows if we show that there exists a constant $c \in (0, 1]$ such that 
%	\begin{equation}
%	\bigg\| \Big|c^{-1} \,\eta_{\nu, R} \ast f_{\nu}\, \frac{ \chi_P}{\phi(P)}\Big|^{\q} \mid L_{\frac{\p}{\q}}(\rn)\bigg\| \leq \sup_{R\in \mathcal{Q}} \bigg\| \Big|\frac{f_\nu \, \chi_R}{\phi(R)}\Big|^{\q} \mid L_{\frac{\p}{\q}}(\rn)\bigg\| + 2^{\nu}=: \delta_\nu, 
%	\end{equation}
%	provided that the supremum on the right-hand side is at most one. 
\end{proof}

\begin{remark}
	In \cite[Lemma~2.2]{ZCY19} and \cite[Lemma~3.12]{WYY18} the authors stated a similar result, but with a stronger condition on the parameter $R$. Using a different decomposition of each function $f_\nu$, we obtained here an extended version of those results. \\
\end{remark}
Lastly, we present a discrete convolution inequality, which extends \cite[Lemma~5.6]{WYY18} slightly.  
\begin{lemma}\label{conv-ineq} 
		Let $p,q\in \mathcal{P}(\rn)$ and $\phi\in \mathcal{G}(\real^{n+1}_+)$. Let $D_1,D_2\in (0,\infty)$ with $D_2 > \max\left\{0, \log_2 \tilde{c}_1(\phi)\right\}$.  For any sequence $(g_\nu)_{\nu\in \nat_0}$ of  measurable functions on $\rn$, consider
		$$
		G_j(x):=\sum_{\nu=0}^j 2^{-(j-\nu)D_2}g_{\nu}(x)+\sum_{\nu=j+1}^{\infty} 2^{-(\nu-j)D_1}g_{\nu}(x), \quad x\in\rn, \quad j\in \no.
		$$
	Then there exist constants $C_1, C_2 >0$, depending on $\p, \q$ and $\delta$, such that
	\begin{equation} \label{disc-B}
	\|(G_{j})_{j \in \nat_0} \mid \SeqB \| \leq C_1 \, \|(g_\nu)_{\nu \in \nat_0} \mid \SeqB\| 
	\end{equation}
	and
	\begin{equation}\label{disc-F}
	\|(G_{j})_{j \in \nat_0} \mid \SeqF \| \leq C_2 \, \|(g_\nu)_{\nu \in \nat_0} \mid \SeqF\|. 
	\end{equation}
\end{lemma}
\begin{proof}
	The inequality \eqref{disc-F} was stated and proved in \cite[Lemma~3.5]{GM18}. Therefore, we are left to prove \eqref{disc-B}.
	
	Let us assume that $p, q \geq 1$. The extension for all $p, q \in \mathcal{P}(\rn)$ can be done similarly as in Step 2 of the proof of \cite[Lemma~3.5]{GM18} and we omit it here. 
		
	Assume that $\|(g_\nu)_{\nu \in \nat_0} \mid \SeqB\|=1$, which, by Definition~\ref{mixed-phi} and Remark~\ref{rmk:mixed}(iii), implies that, for any cube $P \in \mathcal{Q}$,
	\begin{equation}\label{assump}
	\varrho_{\ell_{\q}(L_{\p})} \bigg( \frac{g_\nu \chi_P}{\phi(P)}\bigg) \leq 1.
	\end{equation}
	We then fix a cube $P \in \mathcal Q$ arbitrarily. Hence, 
	\begin{align*}
	I(P) &:= \frac{1}{\phi(P)} \|(G_j)_{j \geq (j_P \vee 0)} \mid \ellqp \|\\
	& \leq  \frac{1}{\phi(P)} \bigg\|\sum_{\nu=0}^j 2^{-(j-\nu)D_2}\ g_{\nu} \, \chi_P \mid \ellqp\bigg\|\\
	&\qquad + \frac{1}{\phi(P)} \bigg\|\sum_{\nu=j+1}^{\infty} 2^{-(\nu-j)D_1}\, g_{\nu} \, \chi_P \mid \ellqp\bigg\|\\
	&= I_2(P) + I_1(P).
	\end{align*}
	
	In what follows we will use the notation $\displaystyle c(\varepsilon):= \sum_{l=0}^\infty 2^{-l \varepsilon}$, for some $\varepsilon>0$. \\
	
	{\it Step 1.} Firstly we show that there exists some constant $c_1>0$ such that $I_1(P) \leq c_1$. Due to the unit ball property, we turn to the modular and note that
	\begin{align}\label{inf}
	&\varrho_{\ell_{\q}(L_{\p})}\bigg( \frac{c_1^{-1}}{\phi(P)} \sum_{\nu=j+1}^{\infty} 2^{-(\nu-j)D_1}\, g_{\nu} \, \chi_P  \bigg)\nonumber \\
	& = \sum_{j=(j_P \vee 0)} \inf\bigg\{ \lambda \in (0, \infty): \bigg\|\frac{c_1^{-1}}{\phi(P)} \sum_{\nu=j+1}^{\infty} 2^{-(\nu-j)D_1}\, \frac{g_{\nu} \, \chi_P}{\lambda^{1/\q}}  \mid \Lp \bigg\|\leq 1 \bigg\}\nonumber\\
	&\leq \sum_{j=(j_P \vee 0)} \inf\bigg\{ \lambda \in (0, \infty): c_1^{-1} \sum_{\nu=j+1}^{\infty} 2^{-(\nu-j)D_1} \Big\| \frac{g_{\nu} \, \chi_P}{\phi(P) \lambda^{1/\q}}  \mid \Lp \Big\|\leq 1 \bigg\}. 
	\end{align}	
	Let us define
	$$
	I_1^{j, \nu}(P):= \inf \bigg\{ \lambda \in (0, \infty): c_1^{-1}\, c(\varepsilon) \, 2^{-(\nu-j)(D_1-\varepsilon)} \Big\| \frac{g_{\nu} \, \chi_P}{\phi(P) \lambda^{1/\q}}  \mid \Lp \Big\|\leq 1 \bigg\},
	$$
	for $j \in \nat_0$, $\nu \geq j+1$ and some $0<\varepsilon<D_1$. We claim that, for each $\nu \geq j+1$, the sum $\displaystyle \sum_{k=j+1}^\infty I_1^{j,k}$ is not smaller than the infimum in \eqref{inf}. We may assume that this sum is finite. For any $\delta>0$ we have 
	$$
	c_1^{-1}\, c(\varepsilon) \, 2^{-(\nu-j)(D_1-\varepsilon)} \Big\| \frac{g_{\nu} \, \chi_P}{\phi(P)\,  [I_1^{j, \nu} + \delta 2^{-\nu}]^{1/\q}}  \mid \Lp \Big\| \leq 1,
	$$
	so that
	$$
	c_1^{-1}\, c(\varepsilon) \sum_{\nu=j+1}^\infty 2^{-(\nu-j)D_1} \Big\| \frac{g_{\nu} \, \chi_P}{\phi(P)\,  [I_1^{j, \nu} + \delta 2^{-\nu}]^{1/\q}}  \mid \Lp \Big\| \leq \sum_{\nu=j+1}^\infty 2^{-(\nu-j)\varepsilon} \leq \sum_{l=0}^\infty 2^{-l \varepsilon}. 
	$$
	Therefore
	$$
	c_1^{-1} \sum_{\nu=j+1}^\infty 2^{-(\nu-j)D_1} \Big\| \frac{g_{\nu} \, \chi_P}{\phi(P)\,  \sum_{k=j+1}^\infty[I_1^{j, k} + \delta 2^{-k}]^{1/\q}}  \mid \Lp \Big\| \leq 1
	$$
	and so
	$$
	\inf\bigg\{ \lambda \in (0, \infty): c_1^{-1} \sum_{\nu=j+1}^{\infty} 2^{-(\nu-j)D_1} \Big\| \frac{g_{\nu} \, \chi_P}{\phi(P) \lambda^{1/\q}}  \mid \Lp \Big\|\leq 1 \bigg\} \leq \sum_{k=j+1}^\infty \big(I_1^{j,k}(P) + \delta 2^{-k} \big).
	$$
	The claim follows then by the convergence of the second part of the series on the right-hand side and the arbitrariness of $\delta>0$. Now using this in \eqref{inf}, we have
	\begin{align*}
		&\varrho_{\ell_{\q}(L_{\p})}\bigg( \frac{c_1^{-1}}{\phi(P)} \sum_{\nu=j+1}^{\infty} 2^{-(k-j)D_1}\, g_{k} \, \chi_P  \bigg)\nonumber \\
		& \leq \sum_{j=(j_P \vee 0)}^\infty \sum_{k=j+1}^\infty \inf \bigg\{ \lambda \in (0, \infty): c_1^{-1}\, c(\varepsilon) \, 2^{-(k-j)(D_1-\varepsilon)} \Big\| \frac{g_{k} \, \chi_P}{\phi(P) \lambda^{1/\q}}  \mid \Lp \Big\|\leq 1 \bigg\}\\
		& \leq c_1^{-1}\, c(\varepsilon) \sum_{k=(j_P \vee 0)}^\infty \sum_{j = (j_P \vee 0)}^k 2^{-(k-j)(D_1-\varepsilon)} \inf \bigg\{ \lambda \in (0, \infty):   \Big\| \frac{g_{k} \, \chi_P}{\phi(P) \lambda^{1/\q}}  \mid \Lp \Big\|\leq 1 \bigg\}\\
		& \leq  \sum_{k=(j_P \vee 0)}^\infty \inf \bigg\{ \lambda \in (0, \infty):   \Big\| \frac{g_{\nu} \, \chi_P}{\phi(P) \lambda^{1/\q}}  \mid \Lp \Big\|\leq 1 \bigg\}\\
		&\leq 1,
	\end{align*}
	doing a convenient change of variables and with the choice of $ c_1 = c(\varepsilon) \, c(D_1-\varepsilon)$. The first part is then proved. \\
	
	{\it Step 2.} We prove now that $I_2(P)\leq c_2$, for some $c_2>0$. We proceed similarly as before, and get, for $0<\varepsilon<\min\{D_2, D_2 - \log_2 \tilde{c}_1(\phi)\}$,
	\begin{align*}
	&\varrho_{\ell_{\q}(L_{\p})}\bigg( \frac{c_2^{-1}}{\phi(P)} \sum_{\nu=0}^{j} 2^{-(j-k)D_2}\, g_{k} \, \chi_P  \bigg)\nonumber \\
	& \leq \sum_{j=(j_P \vee 0)}^\infty \sum_{k=0}^j \inf \bigg\{ \lambda \in (0, \infty): c_2^{-1}\, c(\varepsilon) \, 2^{-(j-k)(D_2-\varepsilon)} \Big\| \frac{g_{k} \, \chi_P}{\phi(P) \lambda^{1/\q}}  \mid \Lp \Big\|\leq 1 \bigg\}\\
	& \leq \sum_{k=0}^\infty \sum_{j = (j_P \vee 0\vee k)}^\infty c_2^{-1}\, c(\varepsilon) \, 2^{-(j-k)(D_2-\varepsilon)} \inf \bigg\{ \lambda \in (0, \infty):   \Big\| \frac{g_{k} \, \chi_P}{\phi(P) \lambda^{1/\q}}  \mid \Lp \Big\|\leq 1 \bigg\}\\
	& = \sum_{k=0}^{(j_P \vee 0)-1} \sum_{j = (j_P \vee 0)}^\infty \quad ... \quad + \quad  \sum_{k=(j_P \vee 0)}^{\infty} \sum_{j = k}^\infty \quad ...\\
	&=: H_1(P) + H_2(P)
	\end{align*}
	for the same $c(\varepsilon)$ as before. For $H_2(P)$, after a proper change of variables, and choosing $c_2~\geq \bar{c}:=~2\,c(\varepsilon)\, c(D_2-\varepsilon)$, we have
	\begin{align*}
	H_2(P) & = c_2^{-1}\, c(\varepsilon) \sum_{k=(j_P \vee 0)}^{\infty} \sum_{j = k}^\infty  2^{-(j-k)(D_2-\varepsilon)} \inf \bigg\{ \lambda \in (0, \infty):   \Big\| \frac{g_{k} \, \chi_P}{\phi(P) \lambda^{1/\q}}  \mid \Lp \Big\|\leq 1 \bigg\}\\
	& =c_2^{-1}\, c(\varepsilon) \, c(D_2-\varepsilon) \sum_{k=(j_P \vee 0)}^{\infty} \inf \bigg\{ \lambda \in (0, \infty):   \Big\| \frac{g_{k} \, \chi_P}{\phi(P) \lambda^{1/\q}}  \mid \Lp \Big\|\leq 1 \bigg\}\\	
	& \leq \frac{1}{2} \, \varrho_{\ell_{\q}(L_{\p})}\bigg(  \frac{g_k \chi_P}{\phi(P)}\bigg)\\
	& \leq \frac{1}{2}.
	\end{align*}
	
	As for $H_1(P)$, we proceed similarly. We then get
	\begin{align*}
	&H_1(P)  = c_2^{-1}\, c(\varepsilon) \sum_{k=0}^{(j_P \vee 0)-1} \sum_{j = (j_P \vee 0)}^\infty  2^{-(j-k)(D_2-\varepsilon)} \inf \bigg\{ \lambda \in (0, \infty):   \Big\| \frac{g_{k} \, \chi_P}{\phi(P) \lambda^{1/\q}}  \mid \Lp \Big\|\leq 1 \bigg\}\\
	&=c_2^{-1}\, c(\varepsilon) \sum_{k=0}^{(j_P \vee 0)-1} 2^{k(D_2-\varepsilon)} \sum_{j = (j_P \vee 0)}^\infty  2^{-j(D_2-\varepsilon)} \inf \bigg\{ \lambda \in (0, \infty):   \Big\| \frac{g_{k} \, \chi_P}{\phi(P) \lambda^{1/\q}}  \mid \Lp \Big\|\leq 1 \bigg\}\\
	& =\frac{c_2^{-1}\, c(\varepsilon)}{1-2^{-(D_2-\varepsilon)}} \sum_{k=(j_P \vee 0)}^{\infty} 2^{-((j_P \vee 0) -k)(D_2-\varepsilon)}\frac{1}{\phi(P)} \inf \bigg\{ \lambda \in (0, \infty):   \Big\| \frac{g_{k} \, \chi_P}{ \lambda^{1/\q}}  \mid \Lp \Big\|\leq 1 \bigg\}\\	
	& \leq \frac{c_2^{-1}\, c(\varepsilon)}{1-2^{-(D_2-\varepsilon)}} \sum_{k=0}^{(j_P \vee 0)-1} 2^{-((j_P \vee 0) -k)(D_2-\varepsilon-\log_2 \tilde{c}_1(\phi))} \varrho_{\ell_{\q}(L_{\p})}\bigg(  \frac{g_k \chi_R}{\phi(R)}\bigg)\\
	& \leq \frac{1}{2}, 
	\end{align*}
	with $R=Q(c_P, 2^{(\j_P \vee 0) - k -j_P})$ and choosing $\displaystyle c_2 \geq \tilde{c}:= \frac{2 \, c(\varepsilon) \, c(D_2-\varepsilon-\log_2 \tilde{c}_1(\phi))}{1-2^{-(D_2-\varepsilon)}}$. In the fourth step we have used \eqref{phi-cond1} and the fact that
	$$
	\inf \bigg\{ \lambda \in (0, \infty):   \Big\| \frac{g_{k} \, \chi_P}{ \lambda^{1/\q}}  \mid \Lp \Big\|\leq 1 \bigg\} \leq \inf \bigg\{ \lambda \in (0, \infty):   \Big\| \frac{g_{k} \, \chi_R}{ \lambda^{1/\q}}  \mid \Lp \Big\|\leq 1 \bigg\}.	
	$$
	We then obtain $I_2(P)\leq c_2$ by considering $c_2 \geq \max\{\bar{c}, \tilde{c}\}$, as we wanted to prove.  
	
\end{proof}
\begin{remark} Naturally, this statement holds also true if the indices $k$ and $\nu$ run only over natural numbers.
\end{remark}

\medskip

\subsection{Variable 2-microlocal Besov-type and Triebel-Lizorkin-type spaces}

We will present now the definition of the spaces under consideration in this paper. To do this, we start by introducing the notions of \textit{admissible weight sequence} and \textit{admissible pair of functions}. 

\begin{definition}\label{def-ad-weight} Let $\alpha\geq 0$ and $\alpha_1,\alpha_2\in\real$ with $\alpha_1\leq \alpha_2$. A
	sequence of non-negative measurable functions in $\rn$ $\bm{w}=(w_j)_{j\in\nat_0}$ belongs to the class $\mathcal{W}^{\alpha}_{\alpha_1,\alpha_2}(\rn)$ if the following conditions are satisfied:
	\begin{list}{}{\labelwidth1.3em\leftmargin2.3em}
		\item[{\upshape (i)\hfill}] There exists a constant $c>0$ such that
		$$
		0<w_j(x)\leq c\,w_j(y)\,(1+2^j|x-y|)^{\alpha}\quad \mbox{for all} \;\, j\in\nat_0 \;\; \mbox{and all} \;\, x,y\in\rn.
		$$
		\item[{\upshape (ii)\hfill}] For all $j\in\nat_0$ it holds
		$$
		2^{\alpha_1}\,w_j(x)\leq w_{j+1}(x)\leq 2^{\alpha_2}\,w_j(x) \quad \mbox{for all}\;\, x \in \rn.
		$$
	\end{list}
	Such a system $(w_j)_{j\in\nat_0}\in\mathcal{W}^{\alpha}_{\alpha_1,\alpha_2}(\rn)$ is called admissible weight sequence.\\
\end{definition}

Properties of admissible weights may be found in \cite[Remark~2.4]{Kem08}.

\medskip
 \begin{definition}
 	We say that a pair $(\varphi, \varphi_0)$ of functions in $\mathscr{S}(\rn)$ is admissible if
 	\begin{equation} \label{eq:adm-pair1}\supp \widehat{\varphi} \subset \{ \xi \in \rn: \frac12 \leq |\xi|\leq 2\}\hbox{ \;and \;}|\widehat{\varphi}(\xi)|>0 \mbox{\; when \;}\frac35\leq |\xi|\leq \frac53
 	\end{equation}
 	and
 	\begin{equation}\label{eq:adm-pair2}
 	\supp \widehat{\varphi}_0 \subset \{ \xi \in \rn: |\xi|\leq 2\}\mbox{ \; and \;}|\widehat{\varphi}_0(\xi)|>0\hbox{\; when \;}|\xi|\leq \frac53.
 	\end{equation}
 	Further, we set $\varphi_j(x):= 2^{jn}\varphi(2^j x)$ for $j \in \nat$. Then $(\varphi_j)_{j \in \nat_0} \subset \mathscr{S}(\rn)$ and
 	\begin{equation}
 		\supp \varphi_j \subset \{ x \in \rn: 2^{j-1} \leq |x| \leq 2^{j+1}\}.\nonumber
 	\end{equation}
 \end{definition}

We are finally in a position to introduce variable 2-microlocal Besov-type and Triebel-Lizorkin-type spaces. 

\begin{definition} \label{def-spaces}
Let $(\varphi, \varphi_0)$ be a pair of admissible functions on $\rn$. Let $p, q \in \Plog$, $\bm{w}=(w_j)_{j\in\nat_0} \in \mathcal{W}^{\alpha}_{\alpha_1,\alpha_2}(\rn)$ and $\phi\in \mathcal{G}(\real^{n+1}_+)$. 
\begin{list}{}{\labelwidth1.3em\leftmargin2.3em}
	\item[{\upshape (i)\hfill}] The variable 2-microlocal Besov-type space $\Bphi$ is defined to be the set of all $f \in \SSn$ such that 
	$$
	\|f \mid \Bphi\| :=  \left\| \big( w_j(\cdot) (\varphi_j \ast f)(\cdot) \big)_{j\in \nat_0}  \mid \SeqB \right\|<\infty. \notag 
	$$
	%where the supremum is taken over all dyadic cubes $P \in \rn$.
	\item[{\upshape (ii)\hfill}] Assume $p^+, q^+< \infty$. The variable 2-microlocal Triebel-Lizorkin-type space $\Fphi$ is defined to be the set of all $f \in \SSn$ such that 
		$$
	\|f \mid \Fphi\| :=  \left\| \big( w_j(\cdot) (\varphi_j \ast f)(\cdot) \big)_{j\in \nat_0}  \mid \SeqF \right\|<\infty. \notag 
	$$
\end{list}
\end{definition}
\begin{remarks} \label{rmk-examples}
\begin{list}{}{\labelwidth1.3em\leftmargin2.3em}
	\item[{\upshape (i)\hfill}] These spaces were introduced by Wu et al. in \cite{WYY18}, where the authors have proved the independence of the spaces on the admissible pair. 
	\item[{\upshape (ii)\hfill}] In the particular case of $w_j(\cdot)=2^{j s(\cdot)}$, with $s \in C^{\log}_{\rm{loc}}(\rn)$, we recover $A_{\p,\q}^{\s, \phi}(\rn)$, $A \in \{B, F\}$, introduced in \cite{YZY15-B} and \cite{YZY15-F} and also investigated in \cite{GM18}. 	
	\item[{\upshape (iii)\hfill}]  When $\phi\equiv 1$,  then  $A_{\p,\q}^{\bm{w}, \phi}(\rn)= A_{\p,\q}^{\bm{w}}(\rn)$,  $A \in \{B, F\}$, are the 2-microlocal Besov and Triebel-Lizorkin spaces with variable exponents. For a good overview on this scale we recommend \cite{AC16} and \cite{AC16-1}.
	\item[{\upshape (iv)\hfill}]  When $\p=p$, $\q=q$ are constant exponents, $w_j(\cdot)=2^{js}, j \in \nat_0$, and $\phi(Q):=|Q|^{\tau}$  for all cubes $Q$ and $\tau\in[0,\infty)$, then  $A_{\p,\q}^{\bm{w}, \phi}(\rn)= A_{p,q}^{s,\tau}(\rn)$, $A \in \{B, F\}$, are the Besov-type and Triebel-Lizorkin-type spaces introduced by Yuan et al. in \cite{YSY10}.\\ 
\end{list}
\end{remarks}

%We finish this section by presenting a result which was proved in \cite[Lemma~19]{KV12}. It stablish a connection between the admissible weight sequence and the $\eta$-function and it will be useful later on.
%\begin{lemma}\label{eta-w}
%	Let $\bm{w}=(w_\nu)_{\nu\in\nat_0}\in\mathcal{W}^{\alpha}_{\alpha_1,\alpha_2}(\rn)$ and let $d\geq\alpha$. Then
%	\[
%	w_{\nu}(x) \eta_{\nu,R+d}(x-y)\leq c \, w_{\nu}(y)\eta_{\nu,R}(x-y),
%	\]
%	with $c>0$ independent of $x,y\in\rn$ and $\nu\in\nat_0$.
%\end{lemma}
%A useful and direct consequence of this estimate is the possibility of moving the weights inside the convolution as follows:
%\begin{equation}\label{eq:conv-weight}
%w_\nu(x)\cdot \big( \eta_{\nu,R+d}\ast f\big)(x) \leq c \, \big(\eta_{\nu,R}\ast \left(w_\nu(\cdot)f\right)\big)(x) \qquad\text{(for $d\geq \alpha$).}
%\end{equation}

\section{Maximal functions and local means characterization}
Let $(\psi_j)_{j \in \no}$ be a sequence in  $\mathscr{S}(\rn)$. For each $f \in \mathscr{S}'(\rn)$ and $a>0$, the Peetre's maximal
functions were defined by Peetre in \cite{Pee75} by
$$
(\psi_j^{*}f)_a(x):= \sup_{y \in
	\mathbb{R}^n}\frac{|\psi_j \ast f(y)|}{(1+|2^j(x-y)|)^a},
\quad x \in \mathbb{R}^n, \;j \in \no.
$$

In \cite[Theorems~4.5 and 4.7]{WYY18}, the authors proved a characterization of $\Bphi$ and $\Fphi$ using the Peetre's maximal functions, but where the sequence $(\psi_j)_{j \in \nat_0}$ is the same as in Definition \ref{def-spaces}, which is built upon an admissible pair. Here we intend to extend those results, by showing that they still hold if one considers more general pairs of functions. Additionally, we also prove that one can replace the admissible pairs in the definition of $\Bphi$ and $\Fphi$ by more general ones (in terms of equivalent quasi-norms). The main result of this section reads then as follows. 

\begin{theorem} \label{Thm:Peetre}
	%Let $p$, $q$, $s$, $\phi$ as in Definition \ref{def-spaces}
	Let $p, q \in \Plog$,  $\bm{w}=(w_j)_{j\in\nat_0} \in \mathcal{W}^{\alpha}_{\alpha_1,\alpha_2}(\rn)$ and $\phi\in \mathcal{G}(\real^{n+1}_+)$. Let $R\in\no$ with
	$ R>\alpha_2 + \max\{0, \log_2\tilde{c_1}(\phi)\}$, where $\tilde{c_1}(\phi)$ is the constant in \eqref{phi-cond1}, and let $\psi_0,\psi\in \mathscr{S}(\rn)$ be such that
	\begin{equation}  \label{cond-psi-0}
	(D^{\beta} \widehat{\psi} )(0)=0 \quad \text{for} \quad 0\leq |\beta|<R
	\end{equation}
	and
	\begin{equation} \label{cond-psi-1}
	|\widehat{\psi}_0(\xi)|>0 \quad \text{on} \quad \{\xi\in\rn:|\xi|\leq k\varepsilon\},
	\end{equation}
	\begin{equation}  \label{cond-psi-2}
	|\widehat{\psi}(\xi)|>0 \quad \text{on} \quad  \bigl\{\xi\in\rn:\frac{\varepsilon}{2}\leq|\xi|\leq k\varepsilon\bigr\},
	\end{equation}
	for some $\varepsilon>0$ and $k\in]1,2]$. 
	\begin{list}{}{\labelwidth1.3em\leftmargin2.3em}
		\item[{\upshape (i)\hfill}] For 
	$$a>\frac{n}{p^-}+c_{\log}\left(\frac1q\right)+\alpha+ \max\{0, \log_2\tilde{c_1}(\phi)\},$$ 
	we have 
	$$
	\|f \mid \Bphi \|  \sim \big\| \bigl(2^{j s(\cdot)} (\psi_j^{*}f)_a \bigr)_{j\in \nat_0} \mid \SeqB \big\| \sim \big\| \bigl(2^{j s(\cdot)} (\psi_j \ast f) \bigr)_{j\in \nat_0} \mid \SeqB \big\|
	$$
	for all  $f \in \mathscr{S}'(\rn)$.
	\item[{\upshape (ii)\hfill}] Assume $p^+, q^+<\infty$. For 
	$$a>\frac{n}{\min\{p^-,q^-\}}+\alpha+\max\{0, \log_2\tilde{c_1}(\phi)\},$$ 
	we have 
	$$
	\|f \mid \Fphi \|  \sim \big\| \bigl(2^{j s(\cdot)} (\psi_j^{*}f)_a \bigr)_{j\in \nat_0} \mid \SeqF \big\| \sim \big\| \bigl(2^{j s(\cdot)} (\psi_j \ast f) \bigr)_{j\in \nat_0} \mid \SeqF \big\|
	$$
	for all  $f \in \mathscr{S}'(\rn)$.
	\end{list}
\end{theorem}
\begin{remark}\label{remark-1}
	\begin{list}{}{\labelwidth1.3em\leftmargin2.3em}
		\item[{\upshape (i)\hfill}] The conditions \eqref{cond-psi-0} are usually called moment conditions, while \eqref{cond-psi-1} and \eqref{cond-psi-2} are the so-called Tauberian conditions. If $R=0$ then no moment conditions \eqref{cond-psi-0} on $\psi$ are required. 
		\item[{\upshape (ii)\hfill}] The case $\phi\equiv1$ is covered by \cite[Theorem~3.1(ii)]{AC16},  where we can find also a discussion on the importance of having a $k$ in the  conditions  \eqref{cond-psi-1}  and  \eqref{cond-psi-2}  in contrast with the case $k=2$ usually  found in the literature in such type of result, cf. e.g. \cite{Kem09,KV12,Ryc99}. 
		\item[{\upshape (iii)\hfill}] When $\p=p$, $\q=q $ are constant exponents,  $\phi(Q):=|Q|^{\tau}$  for all cubes $Q$ and $\tau\in[0,\infty)$, and  $w_j(x)=2^{js}$ for all $x \in \rn$, $j \in \nat_0$ and  $s \in \real$, then $\log_2\tilde{c_1}(\phi)=n\tau$ and such a characterization with $k=2$ has been already established in the homogeneous case by Yang and Yuan, cf. \cite[Theorem~2.1(ii)]{YY10}. 
		\item[{\upshape (iv)\hfill}]In \cite{WYY18} the authors proved the independence of the spaces from the admissible pair as a consequence of the $\varphi$-transform characterization. The above theorem provides an alternative proof, since an admissible pair satisfies conditions \eqref{cond-psi-1} and \eqref{cond-psi-2} with $\varepsilon=\frac 65$ and $k=\frac{25}{18}$. Moreover, it becomes clear that  $\Bphi$ and $\Fphi$ can be defined using more general pairs than the admissible ones (in the sense of equivalent quasi-norms). We refer in particular to the case stated in Corollary \ref{cor-localmeans} below.
	\end{list}
\end{remark}

\medskip 

\begin{remark} The proof of Theorem \ref{Thm:Peetre} can be carried out following the proof done by Rychkov \cite{Ryc99} in the classical case. Part (ii) is an extension of \cite[Theorem~3.2]{GM18}, where the spaces $F^{\s, \phi}_{\p,\q}(\rn)$ were considered. As noticed in \cite[Remark~3.9]{GM18}, the proof can easily be adapted for the more general scale $\Fphi$. In a similar way, one can prove the result for the variable 2-microlocal Besov-type spaces $\Bphi$. In this case, the discrete convolution inequality stated in Lemma \ref{eta-phi}(i) is of great importance. \\
\end{remark}

Lastly in this section we present an important application of Theorem  \ref{Thm:Peetre}, that is when $\psi_0$ and $\psi$, satisfying \eqref{cond-psi-0}-\eqref{cond-psi-2}, are local means. The name comes from the compact support of  $\psi_0:=k_0$ and $\psi:=k$, which is admitted in the following statement.

\begin{corollary} \label{cor-localmeans} 
	Let $p, q \in \mathcal{P}^{\log}(\rn)$ (with $p^+, q^+<\infty$ in the $F$-case), $\bm{w}=(w_j)_{j\in\nat_0} \in \mathcal{W}^{\alpha}_{\alpha_1,\alpha_2}(\rn)$ and $\phi\in \mathcal{G}(\real^{n+1}_+)$. 
	For given  $N\in\no$ and $d>0$, let  $k_0, k\in \mathscr{S}(\rn)$ with $\supp k_0, \supp k\subset d Q_{0,0}$,
	\begin{equation}\label{loc-means:1}
	(D^{\beta}  \widehat{k})(0) = 0 \quad \text{if} \quad 0\leq |\beta|<N,
	\end{equation}
	$\widehat{k_0}(0)\neq 0$ and $\widehat{k}(x)\neq 0$\  if \ $0<|x|<\varepsilon$, for some $\varepsilon>0$.
	If $N>\alpha_2+\max\{0,\log_2\tilde{c_1}(\phi)\}$, then 
	$$
	\|f \mid \Bphi\| 
	\sim \big\| \bigl(2^{j s(\cdot)} (k_j \ast f) \bigr)_{j \in \nat_0} \mid \SeqB \big\|
	$$
	and
	$$
	\|f \mid \Fphi\| 
	\sim \big\| \bigl(2^{j s(\cdot)} (k_j \ast f) \bigr)_{j \in \nat_0} \mid \SeqF \big\|
	$$
	for all  $f \in \mathscr{S}'(\rn)$. 
\end{corollary}
\begin{proof}
	It is clear the existence of   $k_0, k^0\in \mathscr{S}(\rn)$ with $\supp k_0, \supp k^0\subset d Q_{0,0}$, $\widehat{k_0}(0)\neq 0$ and $\widehat{k^0}(0)\neq 0$. 
	Then, following \cite[11.2]{Tri97} and taking $M\in\no$ with $2M\geq N$ define $k:=\Delta^M k^0$. Since $\widehat{k}(x)=(-\sum_{i=1}^n|x_i|^2)^M\widehat{k^0}(x)$, we immediately have \eqref{loc-means:1} and $\widehat{k}(x)\neq 0$\  if \ $0<|x|<\varepsilon$, for a small enough $\varepsilon>0$. The  rest is a direct consequence of  Theorem \ref{Thm:Peetre}. 
\end{proof}

\section{Non-smooth atomic decomposition}

In \cite{WYY18} the authors obtained a characterization of the spaces $\Bphi$ and $\Fphi$ by smooth atomic decompositions, generalizing previous results obtained in \cite{YZY15-B,YZY15-F,Dri13,Kem10,YSY10} for  the particular cases described in 
Remark \ref{rmk-examples}. We recall the result from \cite{WYY18} and start by defining smooth atoms and appropriate sequence spaces, where we opted here for a different normalization.

\begin{definition} \label{def-SA}
Let $K,L\in\no$. A function $a_Q\in C^K(\rn)$ is called a $[K,L]$-smooth atom centered at  $Q:=Q_{\nu k}\in\mathcal{Q}$, where $\nu\in\no$ and $k\in\zn$, if 
  $$\supp a_Q \subset 3\,Q,$$
\begin{equation} \label{SA-cond2}
 \|a_Q(2^{-\nu} \cdot) \mid C^K(\rn) \| \leq 1,
\end{equation}
and, when $\nu\in\nat$,  
\begin{equation}  \label{SA-cond3}
   \int_{\rn} x^{\gamma} a_Q(x) dx=0,
\end{equation}
for all multi-indices $\gamma\in \no^n$ with $|\gamma|<L$.
\end{definition}

\begin{remark} As usual when $L=0$ no moment conditions are required by \eqref{SA-cond3}.
\end{remark}

\begin{definition} Let $p, q \in \Plog$, $\bm{w}=(w_\nu)_{\nu\in\nat_0} \in \mathcal{W}^{\alpha}_{\alpha_1,\alpha_2}(\rn)$ and $\phi\in \mathcal{G}(\real^{n+1}_+)$. 
	\begin{list}{}{\labelwidth1.3em\leftmargin2.3em}
		\item[{\upshape (i)\hfill}] The sequence space $\bphi$ is defined as the set of all sequences $t:= \{t_Q\}_{Q \in \mathcal{Q}^*} \subset \mathbb{C}$ such that
		\begin{equation*}
		\|t \mid \bphi\| := \Big\| \Big( \sum_{m \in \zn}  w_\nu(2^{-\nu}k) |t_{Q_{\nu m}}| \chi_{Q_{\nu m}} \Big)_{\nu \in \nat_0} \mid \ell_{\q}^{\phi}(L_{\p})\Big\|<\infty.
		\end{equation*}
		%is finite, where the supremum is taken over all dyadic cubes $P$ in $\rn$. 
		\item[{\upshape (ii)\hfill}] Assume $p^+, q^+<\infty$. The sequence space $\fphi$ is defined as the set of all sequences $t:= \{t_Q\}_{Q \in \mathcal{Q}^*} \subset \mathbb{C}$ such that
		\begin{equation*}
		\|t \mid \fphi\| := \Big\| \Big( \sum_{m \in \zn}  w_\nu(2^{-\nu}k) |t_{Q_{\nu m}}| \chi_{Q_{\nu m}} \Big)_{\nu \in \nat_0} \mid L_{\p}^{\phi}(\ell_{\q})\Big\|<\infty.
		\end{equation*}
		% where the supremum is taken over all dyadic cubes $P$ in $\rn$.
	\end{list}
\end{definition}

%The following atomic decomposition characterization was obtained in \cite[Theorem~5.3]{WYY18}. 
%\begin{theorem} \label{Thm-SmoothAD-B}  
%Let $p, q \in \Plog$, $\bm{w}=(w_j)_{j\in\nat_0} \in \mathcal{W}^{\alpha}_{\alpha_1,\alpha_2}(\rn)$ and $\phi\in \mathcal{G}(\real^{n+1}_+)$.
%\begin{list}{}{\labelwidth1.3em\leftmargin2em}
% \item[{\upshape (i)\hfill}]  Let $K, L\in\no$ with
%\begin{equation} \notag \label{cond:K,L}
%    K> \alpha_2 + \log_2 \tilde{c}_1(\phi)\quad \mbox{ and } \quad L>\frac{n}{\min\{1, p^-\}} -n - \alpha_1.
%\end{equation}
%Suppose that $\{a_Q\}_{Q\in \mathcal{Q^*}}$ is a family of $[K,L]$-smooth atoms  and  that $\{t_Q\}_{Q\in \mathcal{Q^*}} \in b_{\p,\q}^{\s, \phi}(\rn)$.
%Then 
%	$f:= \sum_{Q\in \mathcal{Q^*}} t_Q\,a_Q$
%converges in   $\SSn$ and 
%$$
%\|f\mid B_{\p,\q}^{\s, \phi}(\rn)\| \leq c\, \|t \mid b_{\p,\q}^{\s, \phi}(\rn)\|
%$$
%with $c$ being a positive constant independent of $t$.
% \item[{\upshape (ii)\hfill}]  Conversely, if  $f\in B_{\p,\q}^{\s, \phi}(\rn)$, then, for any given $K,L\in\no$, there exists a sequence   $\{t_Q\}_{Q\in \mathcal{Q^*}} \in b_{\p,\q}^{\s, \phi}(\rn)$ and a  sequence  $\{a_Q\}_{Q\in \mathcal{Q^*}}$ of $[K,L]$-smooth atoms  such that $f= \sum_{Q\in \mathcal{Q^*}} t_Q\,a_Q$  in   $\SSn$ and 
%\begin{equation}
%\|t \mid b_{\p,\q}^{\s, \phi}(\rn)\|\leq c\, \|f\mid B_{\p,\q}^{\s, \phi}(\rn)\|  \nonumber
%\end{equation}
%with $c$ being a positive constant independent of $f$.
%\end{list}
%\end{theorem}

%\subsection{Non-smooth atoms and local means}

Next we present the notion of non-smooth atoms already used in \cite{GK16} in the context of 2-microlocal spaces with variable exponents and which were slightly adapted  from \cite{Sch13}. 
Note that the usual parameters $K$ and $L$ are now non-negative real numbers instead of non-negative integer numbers.

\begin{definition} \label{def:non-smooth-atoms}
Let $K, L\geq 0$. A function  $a_Q: \rn \rightarrow \mathbb{C}$ is called a  $[K, L]$-non-smooth atom centered at $Q:=Q_{\nu k} \in \mathcal{Q}$, with $\nu \in \nat_0$ and $k \in \zn$, if
%  \begin{list}{}{\labelwidth1.3em\leftmargin2em}
%    \item[{\upshape (i)\hfill}] $\supp a_Q \subset 3\,Q$,
%    \item[{\upshape (ii)\hfill}] $ \|a_Q(2^{-\nu} \cdot) \mid \mathscr{C}^K(\rn) \| \leq 1$,
%    \item[{\upshape (iii)\hfill}] and for every $\psi \in \mathscr{C}^L(\rn)$ it holds $$\Big| \display \int_{\rn} \psi(x) a_Q(x) dx\Big| \leq c \, 2^{-\nu(L + n)}\| \psi \mid \mathscr{C}^L(\rn)\|.$$
%  \end{list}
\begin{equation}  \label{NSA-cond1}
\supp a_Q \subset 3\,Q,
\end{equation}
\begin{equation}  \label{NSA-cond2}
\|a_Q(2^{-\nu} \cdot) \mid \mathscr{C}^K(\rn) \| \leq 1,
\end{equation}
and for every $\psi \in \mathscr{C}^L(\rn)$ it holds 
\begin{equation}  \label{NSA-cond3}
\Big| \displaystyle \int_{\rn} \psi(x) a_Q(x) dx\Big| \leq c \, 2^{-\nu(L + n)}\| \psi \mid \mathscr{C}^L(\rn)\|.
\end{equation}
\end{definition}

\begin{remark} \label{compare-atoms}
Since $C^k(\rn)\hookrightarrow \mathscr{C}^k(\rn)$ for  $k\in\no$, it is clear that condition \eqref{NSA-cond2}  follows from \eqref{SA-cond2}. Moreover, using a Taylor expansion, \eqref{NSA-cond3}  can be derived from \eqref{SA-cond3} when $L\in\nat$, cf. \cite[Remark~3.4]{Sch13}. Therefore, when $K,L\in\no$, any $[K, L]$-smooth atom  is a $[K, L]$-non-smooth atom. Moreover, both conditions \eqref{NSA-cond2} and \eqref{NSA-cond3}  are ordered in $K$ and $L$, i.e. the conditions are stricter for increasing $K$ and $L$, see  \cite[Remark.~3.4]{Sch13}. 
\end{remark}
%Let $N \geq 0$ be given. We choose $k_0, k \in \mathscr{S}(\rn)$ with compact support - e.g. $\supp k, \supp k_0 \subset e \cdot Q_{00}$ - such that
%$$D^{\alpha} \hat{k}(0) =0 \quad \mbox{for} \quad |\alpha|\leq N \qquad \mbox{ and } \qquad \hat{k_0}(0)\neq 0.$$
%Moreover, let there be an $\varepsilon>0$ such that
%$$\hat{k}(x)\neq 0 \quad \mbox{for} \quad 0<|x|<\varepsilon.$$
%Such a choice is possible (see Triebel). We set $k_j= 2^{jn} k(2^j \cdot)$ for $j \in \nat$.

For the next two auxiliary results we refer to \cite[Lemmas~3.6,3.7]{GK16}.
\begin{lemma}
Let $k_j$ be the local means according to Corollary \ref{cor-localmeans} with $d=3$. Then  $c2^{-jn}\, k_j$ is a non-smooth $[K,L]$-atom centered at $Q_{j 0}$, for some constant $c>0$ independently of $j$ and for arbitrary large $K>0$ and  $L\leq N+1$.
\end{lemma}

\begin{lemma} \label{lemma:estimates}
Let $k_j$ be the local means according to Corollary \ref{cor-localmeans} with $d=3$. Let also $(a_Q)_{Q \in \mathcal{Q^*}}$ be non-smooth $[K,L]$-atoms. Then, with $Q= Q_{\nu,k},  \nu \in \nat_0, k \in \zn$, it holds
  \begin{equation*}
    \left| \int_{\rn} k_j(y) a_Q (x-y) \, dy\right| \leq c\,2^{-(j-\nu)K} \chi(c\, Q)(x), \quad \mbox{ for } j\geq \nu \nonumber
  \end{equation*}
  and
  \begin{equation*}
    \left| \int_{\rn} k_j(x-y) a_{Q}(y) \, dy\right| \leq c\,2^{-(\nu-j)(L+n)} \chi(c\, 2^{\nu-j}Q)(x), \quad \mbox{ for } j< \nu. \nonumber
  \end{equation*}
\end{lemma}

We are now ready to state the main theorems of this section. We start by presenting the result for the spaces $\Fphi$. However, we won't present the proof here, as it can be carried out in the same way as the proof of \cite[Theorem~4.10]{GM18}. Due to the use of an admissible weight sequence $\bm{w} \in \mathcal{W}^{\alpha}_{\alpha_1,\alpha_2}(\rn)$ in our case, we only have to do some minor adjustments.
\begin{theorem} \label{Thm-NonSmooth-F}  
	Let $p, q \in \Plog$ with $p^+, q^+<\infty$, $\bm{w}=(w_j)_{j\in\nat_0} \in \mathcal{W}^{\alpha}_{\alpha_1,\alpha_2}(\rn)$ and $\phi\in \mathcal{G}(\real^{n+1}_+)$.
	\begin{list}{}{\labelwidth1.3em\leftmargin2.3em}
		\item[{\upshape (i)\hfill}]  Let $K, L\geq 0$ with
		\begin{equation} 
		K> \alpha_2 + \max\{0,\log_2 \tilde{c}_1(\phi)\}\quad \mbox{ and } \quad L>\frac{n}{\min\{1, p^-,q^-\}} -n - \alpha_1.
		\end{equation}
		Suppose that $\{a_Q\}_{Q\in \mathcal{Q^*}}$ is a family of $[K,L]$-non-smooth atoms and  that $\{t_Q\}_{Q\in \mathcal{Q^*}} \in \fphi$.
		Then 
		$f:= \sum_{Q\in \mathcal{Q^*}} t_Q\,a_Q$
		converges in   $\SSn$ and 
		$$
		\|f\mid\Fphi\| \leq c\, \|t \mid \fphi\|
		$$
		with $c$ being a positive constant independent of $t$.
		\item[{\upshape (ii)\hfill}]  Conversely, if  $f\in \Fphi$, then, for any given $K,L\geq 0$, there exists a sequence   $\{t_Q\}_{Q\in \mathcal{Q^*}} \in \fphi$ and a  sequence  $\{a_Q\}_{Q\in \mathcal{Q^*}}$ of $[K,L]$-non-smooth atoms  such that $f= \sum_{Q\in \mathcal{Q^*}} t_Q\,a_Q$  in   $\SSn$ and 
		\begin{equation}
		\|t \mid \fphi\|\leq c\, \|f\mid \Fphi\|  \nonumber
		\end{equation}
		with $c$ being a positive constant independent of $f$.
	\end{list}
\end{theorem}

\begin{theorem} \label{Thm-NonSmooth-B}  
Let $p, q \in \Plog$, $\bm{w}=(w_j)_{j\in\nat_0} \in \mathcal{W}^{\alpha}_{\alpha_1,\alpha_2}(\rn)$ and $\phi\in \mathcal{G}(\real^{n+1}_+)$.
\begin{list}{}{\labelwidth1.3em\leftmargin2.3em}
 \item[{\upshape (i)\hfill}]  Let $K, L\geq 0$ with
\begin{equation} \label{cond:K,L}
    K> \alpha_2 + \max\{0,\log_2 \tilde{c}_1(\phi)\}\quad \mbox{ and } \quad L>\frac{n}{\min\{1, p^-\}} -n - \alpha_1.
\end{equation}
Suppose that $\{a_Q\}_{Q\in \mathcal{Q^*}}$ is a family of $[K,L]$-non-smooth atoms and  that $\{t_Q\}_{Q\in \mathcal{Q^*}} \in \bphi$.
Then 
	$f:= \sum_{Q\in \mathcal{Q^*}} t_Q\,a_Q$
converges in   $\SSn$ and 
$$
\|f\mid\Bphi\| \leq c\, \|t \mid \bphi\|
$$
with $c$ being a positive constant independent of $t$.
 \item[{\upshape (ii)\hfill}]  Conversely, if  $f\in \Bphi$, then, for any given $K,L\geq 0$, there exists a sequence   $\{t_Q\}_{Q\in \mathcal{Q^*}} \in \bphi$ and a  sequence  $\{a_Q\}_{Q\in \mathcal{Q^*}}$ of $[K,L]$-non-smooth atoms  such that $f= \sum_{Q\in \mathcal{Q^*}} t_Q\,a_Q$  in   $\SSn$ and 
\begin{equation}
\|t \mid \bphi\|\leq c\, \|f\mid \Bphi\|  \nonumber
\end{equation}
with $c$ being a positive constant independent of $f$.
\end{list}
\end{theorem}

\begin{proof}
{\it Step 1.} We start by proving (ii) for what we fix $K,L\geq 0$  and assume that $f\in \Bphi$ is given. 
Then, by \cite[Theorem~5.3]{WYY18}, we know that $f$ can be written as an atomic decomposition with $[K_1,L_1]$-smooth atoms with  
$K_1,L_1\in\no$ chosen so that $K_1\geq K$ and $L_1\geq L$.  Since those atoms are $[K,L]$-non-smooth atoms, cf. Remark \ref{compare-atoms},  part (ii) is proved.

\smallskip

{\it Step 2.}  In this  step we show  that $f= \sum_{Q\in \mathcal{Q^*}} t_Q\,a_Q$ converges in $\mathscr{S}'(\rn)$ if  
 $\{t_Q\}_{Q\in \mathcal{Q^*}} \in \bphi$ and $\{a_Q\}_{Q\in \mathcal{Q^*}}$ is a family of  $[K,L]$-non-smooth atoms  with $K,L \geq 0$ such that \eqref{cond:K,L} holds. To this end, it suffices to show that
 \begin{equation} \label{conv:aim}
    \lim_{N\rightarrow \infty, \Lambda \rightarrow \infty} \sum_{\nu=0}^{N} \sum_{k \in \zn, |k|\leq \Lambda} t_{Q_{\nu k}} a_{Q_{\nu k}}
 \end{equation}
 exists in $\mathscr{S}'(\rn)$, and we rely mainly on the proof of \cite[Theorem~4.10]{GM18}, where the corresponding result was proved for the spaces $F_{\p,\q}^{\s, \phi}(\rn)$. By \eqref{cond:K,L}, we know that there exists $r \in (0, \min\{1, p^-, q^-\})$ such that $\alpha_1 + \frac{n}{p^-}(r-1)>-L$. Let 
 $$
 \tilde{p}(x):= \frac{p(x)}{r}\quad \mbox{and}\quad \tilde{w}_j(x):= w_j(x)\, 2^{\frac{n}{p(x)}(r-1)}, \quad \mbox{for all } x \in \rn, j \in \nat_0.
 $$ 
 Then $\tilde{\bm{w}}=(\tilde{w}_j)_{j \in \nat_0}\in \mathcal{W}^{\tilde{\alpha}}_{\tilde{\alpha}_1,\tilde{\alpha}_2}(\rn)$ and
 $$
 \tilde{\alpha}_1:= \alpha_1 + \frac{n}{p^-}(r-1)>-L.
 $$
 According to \cite[Proposition~4.1(i), Remark~4.2(ii)]{WYY18}, we have the following sequence of embeddings
 $$
\bphi \hookrightarrow b_{\p,\infty}^{\bm{w}, \phi}(\rn) \hookrightarrow b_{\tilde{p}(\cdot),\infty}^{\tilde{\bm{w}}, \phi}(\rn).
 $$
 By similar arguments as used in the proof of \cite[Theorem~4.10]{GM18}, we conclude that there exist $\delta_0 > \max\{0, \log_2\tilde{c}_1(\phi)\} $, $a > n $ and $R>0$ as big as we want such that, for all $h \in \Sn$ and $j \in \nat_0$, 

\begin{align}
 \Big| \int_{\rn} &\sum_{k \in \zn, |k|\leq \Lambda} t_{Q_{\nu k}} a_{Q_{\nu k}}(x) h(x) \, dx \Big| \lesssim 2^{-\nu (L+ \tilde{\alpha}_1)} \sum_{j=0}^{\infty} 2^{-j \delta_0}\sum_{i=0}^{\infty} 2^{-i (R-a)} \nonumber \\
 & \qquad \cdot \bigg\| \sum_{k \in \zn}\tilde{w}_{\nu}(2^{-\nu}k) |t_{Q_{\nu k}}| \chi_{Q_{\nu k}} \Big|\, L_{\tilde{p}(\cdot)}(Q(0, 2^{i+j + c_0}))\bigg\| \nonumber\\
 & \lesssim 2^{-\nu (L+ \tilde{\alpha}_1)} \sum_{j=0}^{\infty} 2^{-j \delta_0}\sum_{i=0}^{\infty} 2^{-i (R+\tilde{\alpha}-a)} \phi(Q(0, 2^{i+j + c_0})) \big\|t \mid b_{\tilde{p}(\cdot),\infty}^{\tilde{\bm{w}}, \phi}(\rn) \big\| \nonumber\\
 & \lesssim 2^{-\nu (L+ \tilde{\alpha}_1)} \sum_{j=0}^{\infty} 2^{-j (\delta_0 - \log_2\tilde{c}_1(\phi))}\sum_{i=0}^{\infty} 2^{-i (R + \tilde{\alpha}-a- \log_2 \tilde{c}_1(\phi))} \big\|t \mid b_{\tilde{p}(\cdot),\infty}^{\tilde{\bm{w}}, \phi}(\rn) \big\| \nonumber\\
 & \lesssim  2^{-\nu (L+ \tilde{\alpha}_1)} \big\|t \mid b_{\tilde{p}(\cdot),\infty}^{\tilde{\bm{w}}, \phi}(\rn) \big\|.
\end{align} 
Here we have used the following estimate, coming from the  properties of the class $\mathcal{W}^{\tilde{\alpha}}_{\tilde{\alpha}_1,\tilde{\alpha}_2}(\rn)$,
$$
2^{-\nu \tilde{\alpha}_1} \tilde{w}_{\nu}(2^{-\nu}m) (1+ 2^\nu |x-2^{-\nu}m|)^{\tilde{\alpha}} (1 + |x|)^{\tilde{\alpha}} \gtrsim 2^{-\nu \tilde{\alpha}_1}\tilde{w}_{\nu}(x)  (1 + |x|)^{\tilde{\alpha}}  \gtrsim 1,
$$
where the implicit constants are independent of $x \in \rn$, $\nu \in \nat_0$ and $m \in \zn$. Therefore, since $L > -\tilde{\alpha}_1$, we conclude that the limit of \eqref{conv:aim} exists in $\SSn$.

\smallskip

{\it Step 3.} We deal now with part (i). %For this we rely on the proofs of \cite[Theorem~8]{YZY15-B} and \cite[Theorem~4.10]{GM18}. 

Firstly, as in the proof of the characterization of $F^{\s, \phi}_{\p, \q}(\rn)$ with non-smooth atoms stated in \cite[Theorem~4.10]{GM18}, we will apply the local means characterization proved in the previous section. Let us assume that  $\{t_Q\}_{Q\in \mathcal{Q^*}} \in \bphi$ and that $\{a_Q\}_{Q\in \mathcal{Q^*}}$ is a family of  $[K,L]$-non-smooth atoms  with $K,L \geq 0$ such that \eqref{cond:K,L} holds. Since the convergence of $f= \sum_{Q\in \mathcal{Q^*}} t_Q\,a_Q$ in $\mathscr{S}'(\rn)$ was already shown in the previous step, we are left to the proof of
$$
\|f\mid\Bphi\| \lesssim \|t \mid \bphi\|.
$$
Without loss of generality, we may assume that $\|t \mid \bphi\|=1$ and show that $\|f\mid \Bphi\| \lesssim 1$. \\

%\textbf{Case 1} We first consider the case when $q^+ < \infty$. 

Let  $k_j$, $j\in\no$, be local means  as in Corollary \ref{cor-localmeans}. Then, %For a given dyadic cube  $P \in \mathcal{Q}$  and $j\in\no$, it holds
$$
  k_j \ast f  =  \sum_{\nu=0}^{j} \sum_{k \in \zn} t_{Q_{\nu k}} \, k_j \ast a_{Q_{\nu k} }
  %+ \sum_{\nu=(j_P\vee 0)}^{j} \sum_{k \in \zn} t_{Q_{\nu k}} \, k_j \ast a_{Q_{\nu k} }
  + \sum_{\nu=j+1}^{\infty} \sum_{k \in \zn} t_{Q_{\nu k}} \, k_j \ast a_{Q_{\nu k} }
$$
%where $\sum_{\nu=0}^{(j_P\vee 0)-1} ... = 0$ \, if \, $j_P\leq 0$.

and hence
\begin{align}
\left\| f \mid \Bphi \right\| &\lesssim  \left\| \left\{ \sum_{\nu=0}^{j} \sum_{k \in \zn}w_j \, |t_{Q_{\nu k}}| \, |k_j \ast a_{Q_{\nu k}}| \right\}_{j \in \nat_0} \mid \ell_{\q}^\phi(L_{\p})\right\|  \nonumber \\
%& \quad  + \frac{1}{\phi(P)}  \left\| \left\{ \sum_{\nu=(j_P\vee 0)}^{j} \sum_{k \in \zn}w_j \, |t_{Q_{\nu k}}| \, |k_j \ast a_{Q_{\nu k}}| \right\}_{j \geq (j_P \vee 0)} \mid \ell_{\q}(L_{\p} (P))\right\| \nonumber \\
& \qquad  +  \left\| \left\{ \sum_{\nu=j+1}^{\infty} \sum_{k \in \zn} w_j \,|t_{Q_{\nu k}}| \, |k_j \ast a_{Q_{\nu k}}| \right\}_{j \in \nat_0} \mid \ell_{\q}^\phi(L_{\p}) \right\|  \nonumber \\
& =: I + II. \label{estimate-0}
\end{align}
In what follows, let $r \in (0, \min\{1, p^-, q^-\})$ such that $L > \frac{n}{r} - n - \alpha_1$. %Here we will make use of the unit ball property stated in Remark \ref{rmk:mixed}(iii) and focus on the corresponding modular.

\medskip
{\it Substep 3.1.}
Firstly we show that $I \lesssim 1$. For any given $P \in \mathcal{Q}$, we split the sum in $\nu$ in two parts as follows
\begin{align*}
\frac{1}{\phi(P)} &\left\| \left\{ \sum_{\nu=0}^{j} \sum_{k \in \zn}w_j \, |t_{Q_{\nu k}}| \, |k_j \ast a_{Q_{\nu k}}| \right\}_{j \geq (j_P \vee 0)} \mid \ell_{\q}(L_{\p} (P))\right\|\\
& \lesssim \frac{1}{\phi(P)}  \left\| \left\{ \sum_{\nu=0}^{(j_P\vee 0)-1} \sum_{k \in \zn}w_j^r \, |t_{Q_{\nu k}}|^r \, |k_j \ast a_{Q_{\nu k}}|^r \right\}_{j \geq (j_P \vee 0)} \mid \ell_{\frac{\q}{r}}(L_{\frac{\p}{r}})\right\|^\frac{1}{r}\\
& \quad + \frac{1}{\phi(P)}  \left\| \left\{ \sum_{\nu=(j_P\vee 0)}^{j} \sum_{k \in \zn}w_j^r \, |t_{Q_{\nu k}}|^r \, |k_j \ast a_{Q_{\nu k}}|^r \right\}_{j \geq (j_P \vee 0)} \mid \ell_{\frac{\q}{r}}(L_{\frac{\p}{r}})\right\|^\frac{1}{3} \\
&=: I_1 + I_2,
\end{align*}
and prove that $I_i\lesssim 1$, for $i=1,2$. Observe that $I_1=0$ if $j_P\leq 0$. Therefore, it suffices to consider the case when $j_P>0$. 
 
Moreover, note that, by the properties of the admissible weight sequence, the following holds true
\begin{equation} \label{est-w1}
w_j(x) \lesssim 2^{(j-\nu)\alpha_2}w_\nu(2^{-\nu}k)(1+2^\nu|x-2^{-\nu}k|)^\alpha,
\end{equation} 
for all $x \in \rn$ and $j\geq \nu$. 

\medskip
We first estimate $I_1$. By \eqref{est-w1} and Lemma~\ref{lemma:estimates}, we can follow the same procedure as in Step 2 of the proof of \cite[Theorem~3.8]{YZY15-F} and obtain
\begin{align*}
\sum_{k \in \zn} & w_j(x)^r \, |t_{Q_{\nu k}}|^r \, |(k_j \ast a_{Q_{\nu k}})(x)|^r\\
&\lesssim 2^{-(j-\nu)(K-\alpha_2) r}\sum_{k \in \zn}  w_\nu(2^{-\nu} k)^r\, |t_{Q_{\nu k}}|^r \, (1+2^{\nu}|x - 2^{\nu}k|)^{(\alpha-M) r}\\
& \lesssim 2^{-(j-\nu)(K-\alpha_2) r} \sum_{i=0}^{\infty} 2^{i(\alpha-M+a)r} \eta_{\nu, ar} \ast \left( \left[ \sum_{k \in \zn} w_\nu(2^{-\nu} k) \, |t_{Q_{\nu k}}| \, \chi_{Q_{\nu k}} \, \chi_{Q(c_P, 2^{i-\nu+c_0})} \right]^r\right)(x),
\end{align*}
where $a>n/r$, $c_P$ is the center of $P$ and $c_0 \in \nat$ independent of $x, P, i, \nu$ and $k$. 

Then, by the fact that $\ell_1(L_{\frac{\p}{r}}) \hookrightarrow \ell_{\frac{\q}{r}}(L_{\frac{\p}{r}})$ and by Remark \ref{rmk:conv-Lp}, we can estimate
\begin{align*}
I_1 & \lesssim \frac{1}{\phi(P)}  \left\| \left\{ \sum_{\nu=0}^{j_P-1} \sum_{k \in \zn}w_j^r \, |t_{Q_{\nu k}}|^r \, |k_j \ast a_{Q_{\nu k}}|^r \right\}_{j \geq j_P} \mid \ell_{1}(L_{\frac{\p}{r}})\right\|^\frac{1}{r}\\
&= \frac{1}{\phi(P)}  \Bigg\{ \sum_{j=j_P}^\infty \left\| \sum_{\nu=0}^{j_P-1} \sum_{k \in \zn}w_j^r \, |t_{Q_{\nu k}}|^r \, |k_j \ast a_{Q_{\nu k}}|^r  \mid L_{\frac{\p}{r}}(\rn)\right\|\Bigg\}^\frac{1}{r}\\
&\lesssim   \Bigg\{ \sum_{j=j_P}^\infty \frac{1}{\phi(P)^r} \bigg\| \sum_{\nu=0}^{j_P-1}2^{-(j-\nu)(K-\alpha_2) r} \sum_{i=0}^{\infty} 2^{i(\alpha-M+a)r} \\
&\qquad\qquad \times \eta_{\nu, ar} \ast \left( \left[ \sum_{k \in \zn} w_\nu(2^{-\nu} k) \, |t_{Q_{\nu k}}| \, \chi_{Q_{\nu k}} \, \chi_{Q(c_P, 2^{i-\nu+c_0})} \right]^r\right)\mid L_{\frac{\p}{r}}(\rn) \bigg\|^r\Bigg\}^{1/r}\\
& \lesssim \Bigg\{\sum_{j=j_P}^\infty \frac{1}{[\phi(P)]^r} \sum_{\nu=0}^{j_P-1}2^{-(j-\nu)(K-\alpha_2) r} \sum_{i=0}^{\infty} 2^{i(\alpha-M+a)r} \\
&\qquad\qquad \times \left\| \sum_{k \in \zn} w_\nu(2^{-\nu} k) \, |t_{Q_{\nu k}}| \, \chi_{Q_{\nu k}} \mid L_{\p}(Q(c_P, 2^{i-\nu+c_0})) \right\|^r\Bigg\}^{1/r}\\
& \lesssim \Bigg\{ \sum_{j=j_P}^\infty \sum_{\nu=0}^{j_P-1}2^{-(j-\nu)(K-\alpha_2) r} \sum_{i=0}^{\infty} 2^{i(\alpha-M+a)r}  \frac{\phi(Q(c_P, 2^{i-\nu+c_0}))^r}{[\phi(P)]^r}\Bigg\}^{1/r}\\
& \lesssim \Bigg\{ \sum_{j=j_P}^\infty \sum_{\nu=0}^{j_P-1}2^{-(j-\nu)(K-\alpha_2) r} \sum_{i=0}^{\infty} 2^{i(\alpha-M+a)r} 2^{(i-\nu+j_P)r \log_2 \tilde{c}_1(\phi)}  \Bigg\}^{1/r},
\end{align*}
where we have used \eqref{phi-cond1} in the last step. Now, noting that $K>\alpha_2 + \max\{0,\log_2 \tilde{c}_1(\phi) \}$, we conclude that
\begin{equation}
I_1 \lesssim \Bigg\{\sum_{j=j_P}^{\infty} 2^{j_P \log_2 \tilde{c}_1(\phi)}  2^{-j(K-\alpha_2) r} \sum_{\nu=0}^{j_P-1}2^{\nu(K-\alpha_2 -\log_2 \tilde{c}_1(\phi)) r}\Bigg\}^{1/r} \lesssim 1, \nonumber
\end{equation}
where $M$ is chosen large enough such that $M>\alpha+a + \log_2 \tilde{c}_1(\phi)$. \\ %Similarly we also have $\sum_{j=j_P}^{\infty} (J_{1,j})^{q^+} \lesssim 1$, which leads us to the desired estimate. \\

\medskip

Let us prove now that $I_2 \lesssim 1$. Here we follow Step 3 of the proof of \cite[Theorem~4.10]{GM18} and use the estimate \eqref{est-w1} to obtain that, for fixed $x \in P$, $j \geq (j_P \vee 0)$ and $\nu$ with $(j_P \vee 0) \leq \nu \leq j$,
\begin{align*}
&\sum_{k \in \zn}  w_j(x)^r \, |t_{Q_{\nu k}}|^r \, |(k_j \ast a_{Q_{\nu k}})(x)|^r\\
 &\lesssim \sum_{k \in \zn}w_j(x)^r \, |t_{Q_{\nu k}}|^r \,2^{-(j -\nu)(K-\alpha_2)r} (1 + 2^{\nu}|x - x_{Q_{\nu k}}|)^{(\alpha-M)r}\\
 & \lesssim 2^{-(j-\nu)(K-\alpha_2)r} \,\sum_{i=0}^{\infty} 2^{(j+i)(\alpha-M+a)r} \eta_{\nu, ar} \ast \biggl( \Bigl[\sum_{k \in \zn } |t_{Q_{\nu k}}|w_\nu(2^{-\nu}k) \chi_{Q_{\nu k}} \chi_{Q(c_P, 2^{i+j-j_P+c_0})}  \Bigr]^r\biggr)(x),
\end{align*}
where $a> n/r$, $c_P$ is the center of $P$, $c_0 \in \nat$ is a positive constant independent of $x, P, i, \nu, k,j$. Similarly as above, we have
\begin{align*}
%I_2 & \lesssim \frac{1}{\phi(P)}  \left\| \left\{ \sum_{\nu=(j_P\vee 0)}^{j} \sum_{k \in \zn}w_j^r \, |t_{Q_{\nu k}}|^r \, |k_j \ast a_{Q_{\nu k}}|^r \right\}_{j \geq j_P} \mid \ell_{1}(L_{\frac{\p}{r}})\right\|^\frac{1}{r}\\
I_2 &= \frac{1}{\phi(P)}  \Bigg\{ \sum_{j=(j_P \vee 0)}^\infty \left\| \sum_{\nu=(j_P\vee 0)}^{j} \sum_{k \in \zn}w_j^r \, |t_{Q_{\nu k}}|^r \, |k_j \ast a_{Q_{\nu k}}|^r  \mid L_{\frac{\p}{r}}(\rn)\right\|\Bigg\}^\frac{1}{r}\\
&\lesssim   \Bigg\{ \sum_{j=(j_P \vee 0)}^\infty \frac{1}{\phi(P)^r} \bigg\| \sum_{\nu=(j_P\vee 0)}^{j}2^{-(j-\nu)(K-\alpha_2) r} \sum_{i=0}^{\infty} 2^{(j+i)(\alpha-M+a)r} \\
&\qquad\qquad \times \eta_{\nu, ar} \ast \left( \left[ \sum_{k \in \zn} w_\nu(2^{-\nu} k) \, |t_{Q_{\nu k}}| \, \chi_{Q_{\nu k}} \, \chi_{Q(c_P, 2^{i+j-j_P+c_0})} \right]^r\right)\mid L_{\frac{\p}{r}}(\rn) \bigg\|^r\Bigg\}^{1/r}\\
& \lesssim \Bigg\{\sum_{j=(j_P \vee 0)}^\infty \frac{1}{[\phi(P)]^r} \sum_{\nu=(j_P\vee 0)}^{j}2^{-(j-\nu)(K-\alpha_2) r} \sum_{i=0}^{\infty} 2^{(i+j)(\alpha-M+a)r} \\
&\qquad\qquad \times \left\| \sum_{k \in \zn} w_\nu(2^{-\nu} k) \, |t_{Q_{\nu k}}| \, \chi_{Q_{\nu k}} \mid L_{\p}(Q(c_P, 2^{i+j-j_P+c_0})) \right\|^r\Bigg\}^{1/r}\\
& \lesssim \Bigg\{ \sum_{j=(j_P \vee 0)}^\infty \sum_{\nu=(j_P\vee 0)}^{j} 2^{-(j-\nu)(K-\alpha_2) r} \sum_{i=0}^{\infty} 2^{(i+j)(\alpha-M+a)r}  \frac{\phi(Q(c_P, 2^{i+j-j_P+c_0}))^r}{[\phi(P)]^r}\Bigg\}^{1/r}\\
& \lesssim \Bigg\{ \sum_{j=(j_P \vee 0)}^\infty \sum_{\nu=(j_P\vee 0)}^{j} 2^{-(j-\nu)(K-\alpha_2) r} \sum_{i=0}^{\infty} 2^{(i+j)(\alpha-M+a)r} 2^{(i+j)r \log_2 \tilde{c}_1(\phi)}  \Bigg\}^{1/r},
\end{align*}

%\begin{align*}
% J_{2,j}:= \bigg\| \frac{\chi_P}{\phi(P)^r} &\sum_{\nu=(j_P \vee 0)}^{j} 2^{-(j-\nu)(K-\alpha_2)r} \,\sum_{i=0}^{\infty} 2^{(j+i)(\alpha-M+a)r} \\
% &\times \eta_{\nu, ar} \ast \left( \left[\sum_{k \in \zn }  w_\nu(2^{-\nu}k)\,|t_{Q_{\nu k}}|\, \chi_{Q_{\nu k}} \chi_{Q(c_P, 2^{-i+j-j_P+c_0})}  \right]^r \right) \mid L_{\frac{\p}{r}}(\rn) \bigg\|^{\frac{1}{r}}.
%\end{align*}
%Then we estimate this quantity using Remark \ref{rmk:conv-Lp} and \eqref{phi-cond1}:
%\begin{align*}
%J_{2,j}& \lesssim \bigg\{ \frac{1}{\phi(P)^r} \sum_{\nu=(j_P \vee 0)}^{j} 2^{-(j-\nu)(K-\alpha_2)r} \,\sum_{i=0}^{\infty} 2^{(j+i)(\alpha-M+a)r} \\
%&\qquad \qquad \times \Big\| \sum_{k \in \zn } w_\nu(2^{-\nu}k)\, |t_{Q_{\nu k}}|\, \chi_{Q_{\nu k}}  \bigr) \mid L_{\p}(Q(c_P, 2^{-i+j-j_P+c_0})) \Big\|^r \bigg\}^{\frac{1}{r}}\\
%& \lesssim  \bigg\{ \sum_{\nu=(j_P \vee 0)}^{j} 2^{-(j-\nu)(K-\alpha_2)r} \,\sum_{i=0}^{\infty} 2^{(j+i)(\alpha-M+a)r} \, \frac{\phi(Q(c_P, 2^{-i+j-j_P+c_0}))^r}{\phi(P)^r} \bigg\}^{\frac{1}{r}}\\
%& \lesssim  \bigg\{ \sum_{\nu=(j_P \vee 0)}^{j} 2^{-(j-\nu)(K-\alpha_2)r} \,\sum_{i=0}^{\infty} 2^{(j+i)(\alpha-M+a)r} \, 2^{r(i+j)\log_2 \tilde{c}_1(\phi)} \bigg\}^{\frac{1}{r}}\\
%& = \bigg\{ 2^{-j(K-\alpha_2 -\log_2 \tilde{c}_1(\phi)-\alpha+ M-a)r} \sum_{\nu=(j_P \vee 0)}^{j} 2^{\nu(K-\alpha_2)r} \,\sum_{i=0}^{\infty} 2^{i(\alpha-M+a+ \log_2 \tilde{c}_1(\phi))r} \bigg\}^{\frac{1}{r}}.
%\end{align*}
Choosing $M$ large enough such that $\displaystyle M> \alpha+a+\log_2 \tilde{c}_1(\phi)$ and since $K>\alpha_2 + \max\{0,\log_2 \tilde{c}_1(\phi) \}$, we conclude that 
$$
I_2 \lesssim \bigg\{ \sum_{j=(j_P \vee 0)}^{\infty} 2^{-j(K-\alpha_2 -\log_2 \tilde{c}_1(\phi)-\alpha+ M-a)r} \sum_{\nu=(j_P \vee 0)}^{j} 2^{\nu(K-\alpha_2)r} \bigg\}^{1/r} \lesssim 1,
$$
as we wanted to demonstrate.\\

\medskip

{\it Substep 3.2.} Lastly, we will show $II \lesssim 1$. For $\nu > j$, by the properties of the admissible weight sequence we have, for all $x \in \rn$,
\begin{equation}\label{est-w2}
w_j(x) \lesssim 2^{-(\nu-j)\alpha_1}w_\nu(2^{-\nu}k)(1+2^j|x-2^{-\nu}k|)^\alpha.
\end{equation}
Together with Lemma \ref{lemma:estimates}, this leads us to
\begin{align*}
&\sum_{k \in \zn}  w_j(x)^r \, |t_{Q_{\nu k}}|^r \, |(k_j \ast a_{Q_{\nu k}})(x)|^r\\
&\lesssim 2^{-(\nu-j)(L+n+\alpha_1-\frac{n}{r}) r}\sum_{k \in \zn}  w_\nu(2^{-\nu} k)^r\, |t_{Q_{\nu k}}|^r \, (1+2^{j}|x - 2^{\nu}k|)^{(\alpha-M) r}\\
& \lesssim 2^{-(\nu-j)(L+n+\alpha_1) r} \sum_{i=0}^{\infty} 2^{i(\alpha-M+a)r} \eta_{j, ar} \ast \left( \left[ \sum_{k \in \zn} w_\nu(2^{-\nu} k) \, |t_{Q_{\nu k}}| \, \chi_{Q_{\nu k}} \, \chi_{Q(c_P, 2^{i-j_P+c_0})} \right]^r\right)(x),
\end{align*}
where $a>n/r$, $c_P$ is the center of $P$ and $c_0 \in \nat$ independent of $x, P, i, \nu$ and $k$. We go back to $II$ and, by Lemmas \ref{conv-ineq} and \ref{eta-phi} together with our assumptions, we conclude
\begin{align*}
II & = \bigg\|\sum_{\nu=j+1}^\infty \sum_{k \in \zn}  w_j(x)^r \, |t_{Q_{\nu k}}|^r \, |(k_j \ast a_{Q_{\nu k}})(x)|^r \mid \ell_{\frac{\q}{r}}^\phi(L_{\frac{\p}{r}}) \bigg\|^\frac1r\\
& \lesssim \bigg\{ \sum_{i=0}^{\infty} 2^{i(\alpha-M+a)r}\bigg\|\sum_{\nu=j+1}^\infty 2^{-(\nu-j)(L+n+\alpha_1-\frac{n}{r}) r}\\
& \qquad \qquad \times \eta_{j, ar} \ast \left( \left[ \sum_{k \in \zn} w_\nu(2^{-\nu} k) \, |t_{Q_{\nu k}}| \, \chi_{Q_{\nu k}} \, \chi_{Q(c_P, 2^{i-j_P+c_0})} \right]^r\right) \mid \ell_{\frac{\q}{r}}^\phi(L_{\frac{\p}{r}}) \bigg\|\bigg\}^\frac1r\\
& \lesssim \bigg\{ \sum_{i=0}^{\infty} 2^{i(\alpha-M+a)r}\bigg\|\eta_{\nu, ar} \ast \left( \left[ \sum_{k \in \zn} w_\nu(2^{-\nu} k) \, |t_{Q_{\nu k}}| \, \chi_{Q_{\nu k}} \, \chi_{Q(c_P, 2^{i-j_P+c_0})} \right]^r\right) \mid \ell_{\frac{\q}{r}}^\phi(L_{\frac{\p}{r}}) \bigg\|\bigg\}^\frac1r\\
& \lesssim \bigg\{ \sum_{i=0}^{\infty} 2^{i(\alpha-M+a)r}\bigg\|\sum_{k \in \zn} w_\nu(2^{-\nu} k) \, |t_{Q_{\nu k}}| \, \chi_{Q_{\nu k}} \, \chi_{Q(c_P, 2^{i-j_P+c_0})} \mid \SeqB \bigg\|\bigg\}^\frac1r\\
& \lesssim \bigg\{ \sum_{i=0}^{\infty} 2^{i(\alpha-M+a+\log_2 \tilde{c}_1(\phi))r} \| t \mid \bphi\|^r \bigg\}^\frac1r\\
&\lesssim 1.
\end{align*}
Here $M$ is chosen large enough such that $M > \alpha +a +\log_2 \tilde{c}_1(\phi)$. We have then finished the proof.  

\end{proof}

\section{Pointwise multipliers}
	{Following \cite{Sch13} and also the corresponding results in the variable setting in \cite{GK16, GM18}, in this section we aim to provide a result on pointwise multipliers for the spaces under consideration. More precisely, let $\varphi$ be a bounded function on $\rn$. The question is under which conditions the mapping $f \mapsto \varphi \cdot f$ makes sense and generates a bounded operator in a given space $\Bphi$ or $\Fphi$. This can be answered in a very pleasant way using the non-smooth atomic decomposition from the previous section. For more references on the topic, we refer to \cite{GM18}.  
	
	%For the classical spaces $B^s_{p,q}(\rn)$ and $F^s_{p,q}(\rn)$, Triebel studied this problem in Section 4.2 of \cite{Tri2}, where two different approaches were followed. The first idea used a smooth atomic decomposition of $f$, but requiring the non-existence of moment conditions like \eqref{SA-cond3}, since the moment conditions are in general destroyed by multiplication with $\varphi$. A more general result was then obtained by Triebel with the help of local means. Recently, Scharf has shown in \cite{Sch13} that it is possible to get a very general result on pointwise multipliers using atomic decomposition but now with the non-smooth atoms.
	
	%Regarding the scale of Triebel-Lizorkin type spaces, in \cite[Theorem~6.1]{YSY10} the authors proved a corresponding result for the spaces $F^{s, \tau}_{p,q}(\rn)$, based in the same techniques as Triebel in \cite{Tri2}. Our aim is to extend this result for the scale of variable exponents $F^{\s, \phi}_{\p, \q}(\rn)$. In this direction, we follow \cite{GK16} and \cite{Sch13}, and use the non-smooth atoms to get the desired result. 
	We start by presenting two helpful results proved in \cite{Sch13}. The first shows that the product of two functions in $\mathscr{C}^s(\rn)$ is still a function in this space, as in \cite[Lemma 4.2]{Sch13}. In the second lemma is stated that the product of a non-smooth $[K,L]$-atom with a function $\varphi \in \mathscr{C}^{\rho}(\rn)$ is still a non-smooth $[K,L]$-atom, and represents a slight normalization of \cite[Lemma 4.3]{Sch13}.}

\begin{lemma} \label{lemma:pointwise-1} Let $s\geq 0$. There exists a constant $c>0$ such that for all $f, g \in  \mathscr{C}^s(\rn)$, the product $f \cdot g$ belongs to $\mathscr{C}^s(\rn)$ and it holds
	\begin{equation}
	\|f \cdot g \mid  \mathscr{C}^s(\rn) \| \leq c\, \|f \mid \mathscr{C}^s(\rn) \| \cdot \| g \mid \mathscr{C}^s(\rn) \|.  \nonumber
	\end{equation}
\end{lemma}

%The  result states that the product of a non-smooth $[K,L]$-atom with a function $\varphi \in \mathscr{C}^{\rho}(\rn)$ is still a non-smooth $[K,L]$-atom, and represents a slight normalization of Lemma 4.3 in \cite{Sch13}.

\begin{lemma} \label{lemma:pointwise-2}
	There exists a constant $c$ with the following property: for all $\nu \in \nat_0$, $m \in \zn$, all non-smooth $[K,L]$-atoms $a_{\nu, m}$ with support in $3\, Q_{\nu, m}$ and all $\varphi \in \mathscr{C}^{\rho}(\rn)$ with $\rho\geq \max(K, L)$, the product
	$$c\, \|\varphi \mid \mathscr{C}^{\rho}(\rn)\|^{-1}\cdot \varphi\cdot a_{\nu, m}$$
	is a non-smooth $[K,L]$-atom with support in $3\, Q_{\nu, m}$.
\end{lemma}

The main result of this section can now be presented. We skip the proof, once it can be carried out as the proof of \cite[Theorem 4.3]{GK16}.
%Now we have all the tools we need to prove the main theorem of this section. Since the proof follows exactly as in \cite[Theorem 4.3]{GK16}, we do not present it here. 
\begin{theorem} 
	Let $p$, $q$, $s$, $\phi$ as in Definition \ref{def-spaces}.    
	\begin{list}{}{\labelwidth1.3em\leftmargin2.3em}
		
		\item[{\upshape (i)\hfill}] Let 
		$$\rho>\max\left\{\alpha_2 ,\alpha_2 +\log_2 \tilde{c}_1(\phi),  \frac{n}{\min\{1, p^-\}} -n - \alpha_1\right\}.$$ 
		Then there exists a positive number $c$ such that
		$$
		\|\varphi \cdot f\mid \Bphi\| \leq c\, \,\|\varphi \mid\mathscr{C}^{\rho}(\rn)\|  \cdot \| f\mid \Bphi\| 
		$$
		for all $\varphi \in \mathscr{C}^{\rho}(\rn)$ and all $f\in \Bphi$.
		
		\item[{\upshape (ii)\hfill}] Let 
		$$\rho>\max\left\{\alpha_2 ,\alpha_2 +\log_2 \tilde{c}_1(\phi),  \frac{n}{\min\{1, p^-, q^-\}} -n - \alpha_1\right\}.$$ 
		Then there exists a positive number $c$ such that
		$$
		\|\varphi\cdot f\mid \Fphi\| \leq c\, \,\|\varphi \mid\mathscr{C}^{\rho}(\rn)\|  \cdot \| f\mid \Fphi\| 
		$$
		for all $\varphi \in \mathscr{C}^{\rho}(\rn)$ and all $f\in \Fphi$.
	\end{list}
\end{theorem}

%\begin{remark}
%	In the particular case of the classical Triebel-Lizorkin spaces \linebreak $F^{s}_{p,q}(\rn)$ this result is well-known and coincides with \cite[Corollary~2.8.2]{Tri83}; see  also \cite{Sch13}.
%\end{remark}


\begin{thebibliography}{30}

\bibitem{AC16}
A.~Almeida,  A. Caetano,
On 2-microlocal spaces with all exponents variable.
Nonlinear Anal. 135, 97--119 (2016).

\bibitem{AC16-1}
A.~Almeida, A.~Caetano, 
Atomic and molecular decompositions in variable exponent 2-microlocal spaces and applications. 
{ J. Funct. Anal. } 270, 1888--1921 (2016).

\bibitem{AH10}
A.~Almeida, P.~H{\"a}st{\"o},
Besov spaces with variable smoothness and integrability.
{  J. Funct. Anal.} 258, no. 5, 1628--1655  (2010).

\bibitem{Bony}
J.~M.~Bony,
Second microlocalization and propagation of singularities for semilinear hyperbolic equations.
{  Hyperbolic equations and related topics} (Katata/Kyoto, 1984), Academic Press, Boston, MA, pp. 11--49 (1986).


\bibitem{Die04}
L.~Diening,
Maximal function on generalized Lebesgue spaces $L^{p(\cdot)}$.
{  Math. Inequal. Appl.}  7, no. 2, 245--253 (2004).

\bibitem{DHHR}
L.~Diening, P.~Harjuletho, P.~H{\"a}st{\"o}, M.~R{\r{u}}{\v z}i{\v c}ka,
{  Lebesgue and Sobolev Spaces with variable exponents}.
 Springer-Verlag, 2011.
 
 \bibitem{DHR09}
 L.~Diening, P.~H{\"a}st{\"o}, and S.~Roudenko,
 Function spaces of variable smoothness and integrability.
 {  J. Funct. Anal.} 256, no. 6, 1731--1768 (2009).

%\bibitem{DHR09}
%L.~Diening, P.~H{\"a}st{\"o}, S.~Roudenko,
%Function spaces of variable smoothness and integrability,
% J. Funct. Anal.  256, no. 6, 1731--1768 (2009).

 \bibitem{Dri13}
D.~Drihem,
Atomic decomposition of Besov-type and Triebel-lizorkin-type spaces.
Sci.  China Math. 56,  1073--1086 (2013).
 
\bibitem{ER00}
D.~E.~Edmunds and J.~R\'akosn\'\i k,
Sobolev embeddings with variable exponent.
{Studia Math.} 143, no. 3, 267--293 (2000).

\bibitem{GMN14}
H.~F.~Gon\c{c}alves, S.~D. Moura and J.~S. Neves,
On trace spaces of 2-microlocal spaces.
{ J. Funct. Anal.} 267, 3444--3468 (2014).

\bibitem{GK16}
H.~F.~Gon\c{c}alves, H.~Kempka,
Non-smooth atomic decomposition of 2-microlocal spaces and application to pointwise multipliers.
{ J. Math. Anal. Appl. 434, 1875--1890 (2016).}

\bibitem{GK17}
H.~F.~Gon\c{c}alves and H.~Kempka,
Intrinsic atomic characterization of 2-microlocal spaces with variable exponents on domains.
{ Rev. Mat. Complut.} 30, no. 3, 467--486 (2017).

\bibitem{GKV17} 
H.~F.~Gon\c{c}alves, H.~Kempka and J.~Vyb{\'{\i}}ral, 
Franke-Jawerth embeddings for Besov and Triebel-Lizorkin spaces with variable exponents. 
{ Ann. Acad. Sci. Fenn. Math.} {43}, 1--23 (2018). 

\bibitem{GM18}
H.~F.~Gon\c{c}alves, S.~D.~Moura, 
Characterization of Triebel-Lizorkin type spaces with variable exponents via maximal functions, local means and non-smooth atomic decompositions.
Math. Nachr. 291, 2024--2044 (2018). https://doi.org/10.1002/mana.201700257

\bibitem{Kem08}
H.~Kempka,
Generalized $2$-microlocal {B}esov spaces.
{  PhD thesis, University of Jena, Germany}, 2008.

\bibitem{Kem09}
H.~ Kempka, 
2-microlocal Besov and Triebel-Lizorkin spaces of variable integrability.
Rev. Mat. Complut. 22,  227--251 (2009).
 
 \bibitem{Kem10}
H.~ Kempka, 
Atomic, molecular and wavelet decomposition of 2-microlocal Besov and Triebel-Lizorkin spaces with variable integrability.
Funct. Approx. Comment. Math. 43, 171--208 (2010).

\bibitem{Kem16}
H.~Kempka,
Intrinsic characterization and the extension operator in variable exponent function spaces on special Lipschitz domains. 
{Function Spaces and Inequalities. Springer Proceedings in Mathematics $\&$ Statistics, vol 206. Springer, Singapore. https://doi.org/10.1007/978-981-10-6119-6\_8}
 
\bibitem{KV12}
H.~ Kempka, J. Vyb\'\i ral,
Spaces of variable smoothness and  integrability: Characterisations by local means and ball means of differences.
J. Fourier Anal. Appl. 18, no. 4, 852--891 (2012).

\bibitem{KemVybnorm} H.~Kempka and J.~Vyb{\'{\i}}ral,
{A note on the spaces of variable integrability and summability of {A}lmeida and {H}\"ast\"o}.
{Proc. Amer. Math. Soc.} 141, no. 9, 3207--3212 (2013).

\bibitem{KR91}
O. Kov\'a\v{c}ik, J.~R\'akosn\'ik,
On spaces $L^{p(x)}$ and $W^{k,p(x)}$.
{  Czechoslovak Math. J.}  40(116), no. 4, 592--618 (1991).

%\bibitem{LXY14}
%P. Li, J. Xiao, Q. Yang, 
%Global mild solutions to modified Navier-Stokes equations with small initial data in critical
%Besov-Q spaces, Electron. J. Differential Equations 185 1--37  (2014).

%\bibitem{LSUYY12} 
%Y.~Liang, Y. Sawano, T. Ullrich, D.~Yang, W.~Yuan, 
%New Characterizations of Besov-Triebel-Lizorkin-Hausdorff Spaces Including Coorbits and Wavelets, 
%J. Fourier Anal. Appl. 18, 1067--1111 (2012).

\bibitem{LYYSU13} 
Y.~Liang, D.~Yang, W.~Yuan, Y. Sawano, T. Ullrich,
A new framework for generalized Besov-type and Triebel-Lizorkin type spaces.
Dissertationes Math. (Rozprawy Mat.) 489, 1--114 (2013).

\bibitem{MNS13} 
S.D. Moura, J.S.~Neves, C. Schneider,
On trace spaces of 2-microlocal Besov spaces with variable integrability.
Math. Nachr. 286, no. 11-12, 1240--1254 (2013).
 
\bibitem{Nak93}
E.~Nakai, 
Pointwise multipliers for functions of weighted bounded mean oscillation.
Studia Math. 105, no. 2, 105--119 (1993).
 
\bibitem{Nak06}
E.~Nakai,
The Campanato, Morrey and H\"older spaces on spaces of homogeneous type.
Studia Math. 176, no. 1, 1--19  (2006).
  
 \bibitem{NS12}
E.~Nakai, Y.~ Sawano, 
Hardy spaces with variable exponents and generalized Campanato spaces.
J. Funct. Anal. 262, no. 9, 3665--3748 (2012). 

\bibitem{Or31}
W.~Orlicz,
\"Uber konjugierte Exponentenfolgen.
{  Studia Math.} 3, 200--212 (1931).
 
\bibitem{Pee75}
J.~Peetre, 
On spaces of Triebel-Lizorkin type.
Ark. Mat. 13,  123--130 (1975).
 

\bibitem{Ryc99} 
V.S.~Rychkov,
On a theorem of Bui, Paluszyski, and Taibleson.
 Proc. Steklov Inst. Math. 227, 280--292 (1999).
  
%\bibitem{SYY10} 
%Y. Sawano, D. Yang, W. Yuan, 
%New applications of Besov-type and Triebel-Lizorkin-type spaces, 
%J. Math. Anal. Appl. 363, 73--85 (2010).
  
 \bibitem{Sch13} 
 B. Scharf, 
 Atomic representations in function spaces and applications to pointwise multipliers and diffeomorphisms, a new approach. 
 Math. Nachr. 286, no.2-3,  283--305  (2013).

%\bibitem{SVy13}
%C.~Schneider, J.~Vyb\'\i ral,
%Non-smooth atomic decompositions, traces on Lipschitz domains and pointwise multipliers in function spaces,
%{ J. Funct. Anal.} 264, no. 5, 1197--1237 (2013).

%\bibitem{Skr98}
%L.~Skrzypczak,
%Atomic decompositions on manifolds with boundary geometry,
%{ Forum Math.} 10, no. 1, 19--38 (1998).

\bibitem{TX05}
L.~Tang, J.~Xu,
Some properties of Morrey type Besov-Triebel spaces.
Math. Nachr. 278, 904--917 (2005).

%\bibitem{Tri83}
%H.~Triebel,
%{  Theory of function spaces}, volume~78 of {Monographs in
%  Mathematics}. Birkh\"auser Verlag, Basel, 1983.


%\bibitem{Tri2}
%H.~Triebel,
%{  Theory of function spaces II}, volume~84 of {Monographs in
%  Mathematics}, Birkh\"auser Verlag, Basel, 1992.

\bibitem{Tri97} 
H. Triebel,
Fractals and Spectra.
Birkh\"auser, Basel (1997).

\bibitem{t13} H. Triebel, Local Function Spaces,
Heat and Navier-Stokes Equations, EMS Tracts in Mathematics 20,
European Mathematical Society (EMS). Z\"urich, 2013.
		
\bibitem{t14} H. Triebel, Hybrid Function Spaces,
Heat and Navier-Stokes Equations, EMS Tracts in Mathematics 24,
European Mathematical Society (EMS). Z\"urich, 2015.


\bibitem{TW96}
H.~Triebel, H.~Winkelvoß,
Intrinsic atomic characterization of function spaces on domains.
Mathematische Zeitschrift 221, no. 1, 647--673 (1996). 

% \bibitem{Ull12}
%T. Ullrich, 
%Continuous characterizations of Besov-Lizorkin-Triebel spaces and new interpretations as coorbits,
% J. Funct. Spaces Appl., Art. ID 163213, 47 pp. (2012)

\bibitem{Vyb09}
J.~Vyb{\'{\i}}ral,
Sobolev and {J}awerth embeddings for spaces with variable smoothness and integrability.
{Ann. Acad. Sci. Fenn. Math.} 34, no. 2, 529--544 (2009).
     
\bibitem{YSY10} 
W.~Yuan, W.~Sickel, D.~Yang,
Morrey and Campanato Meet Besov, Lizorkin and Triebel. 
Lecture Notes in Mathematics 2005, Springer-Verlag, Berlin, 2010.

\bibitem{YY10} 
W.~Yuan, D.~Yang,
Characterizations of Besov-type and Triebel-Lizorkin-type spaces via maximal functions and local means. 
Nonlinear Anal. 73, 3805--3820 (2010).

\bibitem{YZY15-F}
D.~Yang, C.~Zhuo, W.~Yuan,
Triebel-Lizorkin type spaces with variable exponent.
{ Banach J. Math. Anal. 9, no. 4, 146--202 (2015).}

\bibitem{YZY15-B}
D.~Yang, C.~Zhuo, W.~Yuan,
Besov type spaces with variable smoothness and integrability.
{ J. Funct. Anal. 269, 1840--1898 (2015).}

\bibitem{WYY18}
S.~Wu, D.~Yang, W.~Yuan, 
Variable 2-Microlocal Besov–Triebel–Lizorkin-Type Spaces.
{ Acta. Math. Sin. - English Ser. 34, 699--748 (2018). https://doi.org/10.1007/s10114-018-7311-7}

\bibitem{ZCY19}
C.~Zhuo, D.~Chang, D.~Yang, 
Ball Average Characterizations of Variable Besov-type Spaces. 
{Taiwanese Journal of Mathematics 23, 427--452 (2019).}


\end{thebibliography}
\end{document}